\documentclass[11pt]{amsart}
\usepackage{amssymb}
\usepackage{amsmath,amssymb}
 \usepackage{Dynkin}
 \usepackage[unicode]{hyperref}
 \usepackage{multirow}
\theoremstyle{plain}
\newtheorem{thm}{Theorem}[section]
\newtheorem{theorem}[thm]{Theorem}

\newtheorem{lemma}[thm]{Lemma}
\newtheorem{corollary}[thm]{Corollary}
\newtheorem{proposition}[thm]{Proposition}
\theoremstyle{definition}
\newtheorem{remark}[thm]{Remark}

\newtheorem{definition}[thm]{Definition}

\newtheorem{example}[thm]{Example}

\numberwithin{equation}{section}
\usepackage{stackrel}

\newcommand{\0}{{\mathcal O}}

\newcommand{\sC}{{\mathcal C}}
\newcommand{\sD}{{\mathcal D}}

\newcommand{\sH}{{\mathcal H}}

\newcommand{\sL}{{\mathcal L}}

\newcommand{\sO}{{\mathcal O}}
\newcommand{\sP}{{\mathcal P}}

\newcommand{\C}{{\mathbb C}}

\newcommand{\BP}{{\mathbb P}}

\newcommand{\tr}{{\rm tr}}
\newcommand{\id}{{\rm Id}}

\newcommand{\End}{{\rm End}}

\newcommand{\bL}{\mathbf{L}}

\newcommand{\bX}{\mathbf{X}}

\newcommand{\fg}{{\mathfrak g}}

\newcommand{\fgl}{{\mathfrak g}{\mathfrak l}}
\newcommand{\ff}{{\mathfrak f}}
\newcommand{\fm}{{\mathfrak m}}
\newcommand{\fp}{{\mathfrak p}}

\newcommand{\fa}{{\mathfrak a}}

\newcommand{\fz}{{\mathfrak z}}

\newcommand{\fh}{{\mathfrak h}}

\newcommand{\fs}{{\mathfrak s}}

\newcommand{\fsp}{{\mathfrak s}{\mathfrak p}}
\newcommand{\aut}{{\mathfrak a}{\mathfrak u}{\mathfrak t}}
\newcommand{\fcsp}{{\mathfrak c}{\mathfrak s}{\mathfrak p}}

\newcommand\Aut{{\rm Aut}}

\newcommand{\Fr}{{\rm Fr}}

\def\Sym{\mathop{\rm Sym}\nolimits}

\def\Hom{\mathop{\rm Hom}\nolimits}

\def\min{\mathop{\rm min}\nolimits}

\newcommand\bbC{\mathbb{C}}
\newcommand\bbP{\mathbb{P}}
\usepackage{graphicx}
\newcommand\Ch{\operatorname{Ch}}
\newcommand\sfO{\mathsf{O}}
\newcommand\sfN{\mathsf{N}}
\newcommand\sfD{\mathsf{D}}
\newcommand\cC{\mathcal{C}}
\newcommand\cG{\mathcal{G}}
\newcommand\fX{\mathfrak{X}}
\usepackage{xcolor}
\usepackage{tikz-cd}
\usetikzlibrary{arrows.meta}

\tikzset{>={Stealth[scale=1.5]}}
\def\greencolor{green!60!black}
\def\graycolor{gray!60}

\newcommand\finf[1]{\mathfrak{aut}(#1)}
\newcommand\pr{\mathfrak{pr}}
\newcommand\fder{\mathfrak{der}}

\newcommand\Heis{\mathbb{H}}
\newcommand\heis{\mathfrak{H}}
\newcommand\CSp{{\rm CSp}}
\newcommand\wZ{\widehat{Z}}
\newcommand\bZ{\mathbf{Z}}
\newcommand\wbZ{\widehat{\bZ}}
\newcommand\GL{{\rm GL}}
\newenvironment{psm}
  {\left(\begin{smallmatrix}}
  {\end{smallmatrix}\right)}
\newcommand\diag{\operatorname{diag}}

\newcommand\cE{\mathcal{E}}
\newcommand\cR{\mathcal{R}}
\newcommand\cI{\mathcal{I}}
\newcommand\cL{\mathcal{L}}
\newcommand\wD{\widetilde{D}}
\newcommand\bbV{\mathbb{V}}
\newcommand\fann{\mathfrak{ann}}
\newcommand\tgr{\mathrm{gr}}
\newcommand\ad{\operatorname{ad}}
\newcommand\sfZ{\mathsf{Z}}

\title[Symmetries of $(2,3,5)$-distributions]{Symmetries of  $(2,3,5)$-distributions and associated Legendrian cone structures}

\author{Jun-Muk Hwang}
\address{Center for Complex Geometry, Institute for Basic Science (IBS), Daejeon 34126, Republic of Korea}
\email{jmhwang@ibs.re.kr}

\author{Dennis The}
\address{Department of Mathematics and Statistics, UiT The Arctic University of Norway, Troms\o{} 90-37, Norway}
\email{dennis.the@uit.no}



\begin{document}

\maketitle

\begin{abstract}
We exploit a natural correspondence  between holomorphic $(2,3,5)$-distributions  and nondegenerate lines on holomorphic contact manifolds of dimension $5$ to present a new perspective in the study of symmetries of $(2,3,5)$-distributions. This leads to a number of new results in this classical subject, including an unexpected relation between the multiply-transitive families of models having $7$- and $6$-dimensional symmetries, and a one-to-one correspondence between equivalence classes of  nontransitive $(2,3,5)$-distributions with $6$-dimensional symmetries and nonhomogeneous nondegenerate Legendrian curves in $\BP^3$.
An ingredient for establishing the former is an explicit classification of homogeneous nondegenerate Legendrian curves in $\BP^3$, which we present.
\end{abstract}

\medskip

\section{Main Results}\label{s.1}
While studying some problems in algebraic geometry involving rational curves, Ngaiming Mok and the first-named author introduced (see \cite{HM99}) the notion of VMRT (abbreviation of Varieties of Minimal Rational Tangents), a special type of cone structures associated with certain families of rational curves on complex manifolds. The geometry of VMRT is reflected in a natural distribution on the space of rational curves (such as the one in Definition \ref{d.D}). For certain classes of rational curves on  complex manifolds of dimension 5, this natural distribution becomes a {\sl $(2,3,5)$-distribution}, i.e.\ a rank 2 distribution $D \subset TM$ on a 5-manifold $M$ with its derived distribution $D^2 = [D,D]$ having rank 3, and $D^3 = [D,D^2] = TM$.

It turns out that this association can be reverted: any $(2,3,5)$-distribution gives rise to a natural family of rational curves on a 5-dimensional holomorphic contact manifold (Theorem 5.10 of \cite{HLGC}). That a $(2,3,5)$-distribution has a naturally associated family of germs of  curves, called {\sl abnormal extremals}, on a 5-dimensional contact manifold  has been known before (\cite{BH} and \cite{Zel1999}). But the discovery that  these abnormal extremals, in the holomorphic setting,  are germs of natural rational curves  establishes a completely canonical 1-to-1 correspondence between  $(2,3,5)$-distributions  and   lines (in the sense of Definition \ref{d.line}) on 5-dimensional holomorphic contact manifolds. This correspondence is described by the following  canonical double fibration, where $D \subset TM$ is a $(2,3,5)$-distribution on a $5$-dimensional complex manifold and $\sC \subset \BP H$ is a Legendrian cone structure of VMRT type on a $5$-dimensional contact manifold $(X, H \subset TX)$. Here, the line distributions $E = \ker(d\mu)$ and $V = \ker(d\rho)$ on $\cC$ are the vertical distributions for the respective submersions to $X$ and $M$ respectively.

\begin{center}
\begin{figure}[h]
\begin{tikzcd}
& (\cC = \bbP D;E,V) \arrow[ld,"\rho"'] \arrow[rd,"\mu"] \\ (M;D) & & (X;H,\cC) \end{tikzcd}
\label{F:fibration-intro}
\caption{Canonical double fibration relating $(2,3,5)$-distributions and Legendrian cone structures of VMRT type}
\end{figure}
\end{center}

This is a generalization of  the ``flat'' double fibration associated to flag varieties of the 14-dimensional exceptional (complex) simple Lie group $G_2$ (Figure \ref{F:fibration-flat}). The fiber of the Legendrian cone structure on $G_2/P_2$ is equal to the rational normal curve in $\BP^3$, to be denoted by $\bZ \subset \BP^3.$

\begin{center}
\begin{figure}[h]
\begin{tikzcd}
& G_2 / P_{1,2} \arrow[ld] \arrow[rd] \\ G_2 / P_1 & & G_2 / P_2 \end{tikzcd}

\caption{Canonical double fibration associated to flag varieties of $G_2$ (see Section \ref{S:Tanaka} for $P_1, P_2, P_{1,2}$ notations)}
\label{F:fibration-flat}
\end{figure}
\end{center}

We note that the fundamental theorem of parabolic geometries \cite{CS2009} establishes a categorical equivalence between so-called {\sl regular, normal} Cartan geometries (modelled on generalized flag varieties) and underlying geometric structures.  In particular, this perspective exists for each of the three flag varieties in Figure \ref{F:fibration-flat}.  The $(2,3,5)$-geometries $(M;D)$ (resp. geometries of  $(\bbP D; E,V)$) in Figure \ref{F:fibration-intro} are precisely the underlying structures of regular, normal parabolic geometries modelled on $(G_2,P_1)$ (resp. $(G_2,P_{1,2})$).  But the Legendrian cone structures $(X;H,\cC)$ of VMRT type in Figure \ref{F:fibration-intro} are different from the underlying structures of regular, normal parabolic geometries modelled on $(G_2, P_2)$. Instead, the latter are $\bZ$-isotrivial Legendrian cone structures (see \cite{The2018} for examples) and they are of VMRT type only when flat, namely, equivalent to the geometry of the flag variety $(G_2, P_2)$.  In fact,  the result of \cite{Cap2005} implies that regular, normal parabolic geometries modelled on the double fibration  in Figure \ref{F:fibration-flat} form a double fibration only in the flat case.  Consequently, working within the standard framework of regular, normal parabolic geometries is too restrictive for our study.

The double fibration  in Figure \ref{F:fibration-intro} leads to a natural correspondence between symmetries of $(2,3,5)$-distributions and symmetries of the associated cone structures.  This correspondence gives us new insights in symmetries of $(2,3,5)$-distributions, an old subject going back to Cartan's work \cite{Car1910}.  The goal of this paper is to pursue  these insights and work out some explicit consequences.
To explain our results, we need the following terminology.

\begin{definition}\label{d.autD}
Let $D$ be a distribution on a complex manifold $M$, namely, a vector subbundle
$D \subset TM$ of the holomorphic tangent bundle $TM$. For an open subset $U \subset M$ and a vector field $\vec{v}$ on $U$, we have a canonical lifting of $\vec{v}$ to a vector field $\vec{v}'$ on $TU$. In fact, the 1-parameter family of local biholomorphisms of $U$ generated by $\vec{v}$ induce a 1-parameter family of local biholomorphisms of $TU$ whose derivatives give the vector field $\vec{v}'$ on $TU$.
\begin{itemize}
 \item[(0)] We say that a statement ${\bf P}$ holds for a {\sl general point} $y \in M$, if there exists a dense open subset $M^o \subset M$, which may depend on ${\bf P}$, such that ${\bf P}$ holds for any point $y \in M^o$.      \item[(i)] A vector field $\vec{v}$ is an {\sl infinitesimal automorphism} of the distribution $D$ if $\vec{v}'$ is tangent to $D|_U \subset TU$.
 \item[(ii)] For a point $y \in M$, denote by $\aut(D)_y$ the Lie algebra of germs of infinitesimal automorphisms of $D$ in neighborhoods of $y$ and by $\aut(D)^0_y$ the subalgebra of elements of $\aut(D)_y$ that vanish at $y$.
 \item[(iii)] We say $D$ is {\sl transitive at} $y \in M$ if  $\aut(D)_y$ generates $T_y M.$ We say that   $D$  is {\sl transitive} if it is transitive at some (hence a general)  point $y  \in M$ and $D$ is {\sl nontransitive} otherwise.
 \item[(iv)] We say $D$ is {\sl  multiply-transitive at} $y \in M$  if it is transitive at $y$ and $\aut(D)_y^0 \neq 0$. We say $D$ is {\sl multiply-transitive} if it is multiply-transitive at a general point $y \in M$.
 \item[(v)] Define
 \begin{eqnarray*}
 \dim \aut(D) &:=& \min_{y \in M}  \dim \aut(D)_y \\
 \dim \aut(D)^0 &:=& \min_{y \in M} \dim \aut(D)^0_y.
\end{eqnarray*}
 \item[(vi)] A point $y\in M$ is $\aut$-{\sl generic}, if $\dim \aut(D) = \dim \aut(D)_y$ and $\dim \aut(D)^0 = \dim \aut(D)_y^0$.   \end{itemize}
  \end{definition}

Our main result is the  following.

\begin{theorem}\label{t.aut}
Let $D$ be a $(2,3,5)$-distribution on $M$. \begin{itemize}
\item[(i)]  If $\dim \aut(D)^0 \geq 1$, then $\dim \aut(D) \geq 5.$
\item[(ii)] If $\dim \aut(D)^0 \geq 2$,   then  $\dim \aut(D) \geq 6$.
\item[(iii)] If  $\dim \aut(D)^0 \geq 3$, then $\dim \aut(D) = 14$.
\item[(iv)] If  $D$ is nontransitive and $\dim \aut(D) = 6$, then  $\dim \aut(D)^0 = 2.$
\end{itemize}
\end{theorem}

In (iii) above, 14 is the maximal symmetry dimension of any $(2,3,5)$-distribution, this is attained locally unique structure (called the {\sl flat} $(2,3,5)$-distribution), and $\aut(D)$ is isomorphic to the Lie algebra of $G_2$. Some of the statements in Theorem \ref{t.aut} can be deduced {\em a posteriori} from the classification results of multiply-transitive $(2,3,5)$-distributions in \cite{Car1910}, \cite{DK2014}, \cite{The2022} (see Table 1 below).  But even for these cases, our arguments  give a new conceptual proof,  independent of classification results.

Fibers of the cone structure $\sC \subset \BP H$ in Figure \ref{F:fibration-intro} are Legendrian curves in $\BP^3$, namely, projective curves whose affine cones in $\C^4$ are Lagrangian with respect to  a symplectic form on $\C^4$.  For each $D$ in Theorem \ref{t.aut}, we have the associated Legendrian curve $Z \subset \BP^3$. For the multiply-transitive cases of (i)--(iii) in  Theorem \ref{t.aut}, the corresponding Legendrian curves are  homogeneous, while the Legendrian curves corresponding to the case (iv) are nonhomogeneous.  The nontransitive structures in (iv) were known to Cartan \cite[p.170, eq.\ (5)]{Car1910}, but  via the canonical double fibration (Figure \ref{F:fibration-intro}), we can  identify their geometric origin
in terms of nonhomogeneous Legendrian curves as follows.

 \begin{theorem}\label{t.1to1}
There is a natural one-to-one correspondence between
the equivalence classes of the germs at  $\aut$-generic points of nontransitive $(2,3,5)$-distributions $D$ with $\dim \aut(D) =6$ and the projective equivalence classes of
germs of nonhomogeneous nondegenerate Legendrian curves in $\BP^3$.
 \end{theorem}


For these reasons, it is worth clarifying the theory of Legendrian projective curves relevant to our study.

 The study of curves in projective spaces under projective equivalence is a classical subject \cite{Lie, Hal1880, Wil1905}.  In Section \ref{S:4ODE}, we recall details pertinent to the case of (germs of) nondegenerate Legendrian curves in $\bbP^3$, including a relative invariant $q_0$ (of weight 4) and an absolute invariant $\cI$ (when $q_0 \neq 0$).  In Section \ref{S:HNLC}, we specialize this to {\em homogeneous} such curves and completely describe in Theorem \ref{T:homZ} the projective equivalence classes $\bL_{r^2} \cong \bL_{1/r^2}$ (for $r^2 \in \C$), invariant classification via $(q_0, \cI)$, and representative curves.  In particular, $q_0 = 0$ distinguishes the rational normal curve $\bZ \subset \bbP^3$.

Application of this classification to curves $Z \subset \bbP^3$ arising from (complex) multiply-transitive $(2,3,5)$-distributions\footnote{As in  \cite{The2022}, each label refers to: (i) the root type of the fundamental Cartan quartic, (ii) the symmetry dimension, (iii) parameter (if relevant).} yields the results of Figure \ref{F:MTcurve}.

 \begin{figure}[h]
\begin{align}
\begin{array}{|c|c|c|c|} \hline
\mbox{$(2,3,5)$-distribution} & q_0 & \cI & \mbox{$(X;H,\cC)$ locally flat?} \\ \hline
\sfO14 & 0 & \cdot & \checkmark\\
\sfN7_c & \neq 0 & -\frac{c^2}{6} & \checkmark\\
\sfN6 &\neq 0 & -\frac{1}{7} & \times\\
\sfD6_a & \neq 0 & \frac{a^2}{36} & \times\\ \hline
\end{array}
\end{align}
\caption{Homogeneous nondegenerate Legendrian curves associated to multiply-transitive $(2,3,5)$-distributions}
\label{F:MTcurve}
\end{figure}

In particular, this yields unexpected relations between the homogeneous $(2,3,5)$-distributions with 7- and 6-dimensional symmetries:
 \begin{itemize}
 \item when $a^2 = -6c^2$, the associated curves $Z \subset \BP^3$ for $\sfN7_c$ and $\sfD6_a$ are projectively equivalent;
 \item if moreover $a^2 = -6c^2 = - \frac{36}{7}$, then the associated curve $Z \subset \BP^3$ for $\sfN6$ is projectively equivalent to the associated curves of $\sfN7_c$ and $\sfD6_a$.
 \end{itemize}

 Another interesting consequence is that the symmetry algebra of the $\sfN7_c$ case should be isomorphic to $\heis \rtimes \bbC^2$ (see Remark \ref{r.Heis}). This has been proved in \cite{DK2014} by explicit Lie algebra computation, but our argument gives a conceptual geometric proof of this fact.

  As an illustration of our results, we examine,
 at the end of Section \ref{S:SymLeg},  the {\sl rolling distribution} associated with two (real) 2-spheres with distinct ratio of radii $\rho > 1$ rolling on each other without twisting or slipping.  This is multiply-transitive, and its complexification has associated Legendrian curve (complex) projectively equivalent to $Z \subset \bbP^3$ arising from $\gamma(t) = \exp(tA) z$, where
 \begin{align}
 A = \diag(\rho,1,-1,-\rho), \quad z = (1,1,1,1)^\top \in \C^4.
 \end{align}
 The $\rho = 3$ case is geometrically distinguished: this is precisely the case of the rational normal curve $\bZ$.   Indeed, the distribution is of type $\sfD6_a$ for some $a$ when $\rho \neq 3$, while it is of type $\sfO14$ when $\rho = 3$.

 Finally, collecting all the ingredients from previous sections, we prove our main results in Section \ref{S:MainResults}.

\section{Legendrian curves in $\bbP^3$}

\subsection{Preliminaries}

 \begin{definition}\label{d.Z+}
 Let $V$ be a complex vector space and $\BP V$ its projectivization.  For any $z \in V \backslash 0$, we write $\widehat{z}$ for the line spanned by $z$, and corresponding element $[z] \in \BP V$.  Let  $Z \subset \BP V$ be a {\sl curve}, i.e.\ a (not necessarily closed) 1-dimensional complex submanifold.
 \begin{itemize}
 \item[(i)] Its {\sl affine cone} is given by $\wZ = \bigcup_{[z] \in Z} \widehat{z}$.
 \item[(ii)] Its {\sl affine tangent space at $z \in \wZ \backslash 0$} is the tangent space $T_z \wZ \subset V$.  We iteratively define its $(i+1)$-st {\sl osculating space} as
 \begin{align}
 T^{(i+1)}_z \wZ := T^{(i)}_z \wZ + \left\{ \gamma^{(i+1)}(0) \mid \gamma \mbox{ is a curve in } \wZ \mbox{ with } \gamma(0) = z \right\},
  \end{align}
  where $T^{(1)}_z \wZ := T_z \wZ$.  This generates the {\sl osculating sequence}
 \begin{align}
 \widehat{z} \,\,\subset \,\, T_z \wZ \,\,\subset \,\, T^{(2)}_z \wZ \,\,\subset \,\, T^{(3)}_z \wZ \,\,\subset \,\,...
 \end{align}

 \item[(iii)]  We say that $[z] \in Z$ is a {\sl nondegenerate point} of $Z$ if $T^{(k)}_z \wZ = V$ for some $k \geq 1$, and $Z$ is {\sl nondegenerate} if $Z$ has a nondegenerate point.

 \end{itemize}
 \end{definition}

 \begin{definition}
 Let  $(V,\sigma)$ be a symplectic vector space, so $\sigma : \wedge^2 V \to \bbC$ is a symplectic form, i.e.\ a nondegenerate skew-symmetric form.  Define the {\sl conformal symplectic group} $\CSp(V) \subset {\rm GL}(V)$ as the (connected) Lie subgroup preserving $\sigma$ up to an overall scaling factor, i.e.
 \begin{equation}
  \begin{aligned}
 \CSp(V) &= \{ A \in \GL(V) \mid \sigma (A(v),A(w)) = \lambda \sigma(v,w), \\
 &\hspace{2in} \forall v,w \in V, \, \exists \lambda \in \C^\times \}.
 \end{aligned}
 \end{equation}
 Its Lie algebra is the {\sl conformal symplectic algebra} $\fcsp(V) \subset \fgl(V)$.
 \end{definition}

 Let us now specialize to the $\dim V = 4$ case.

\begin{definition}\label{d.Legendre} Suppose $(V,\sigma)$ is a 4-dimensional symplectic vector space.  A curve $Z \subset \BP V \cong \bbP^3$ is {\sl Legendrian} if $\wZ \backslash 0$ is a Lagrangian submanifold of $V$ with respect to $\sigma$, i.e.\
 $\sigma(T_z \widehat{Z}, T_z \widehat{Z}) =0$ for each $z \in \wZ \backslash 0$.  Define
 \begin{align}
 \Aut(Z) := \{ A \in \CSp(V) \mid A(z) \in \widehat{Z}, \, \forall z \in \widehat{Z} \} \supset \bbC^\times \id_V.
 \end{align}
 This is a Lie subgroup with Lie algebra
 \begin{align}
 \aut(Z) := \{ u \in \fcsp(V) \mid u(z) \in T_z \widehat{Z},\,\, \forall z \in \widehat{Z} \} \supset \bbC\, \id_V.
 \end{align}
 If $Z$ is nondegenerate, then we say that $Z$ is {\sl homogeneous} if $\dim \finf{Z} \geq 2$ and {\sl nonhomogeneous} if $\dim \finf{Z} =1$.  (Note ${\rm Id}_V$ acts trivially on $Z$.)
 \end{definition}

\begin{example}\label{e.Sym3}
 Let $V = \Sym^3 W$, where $W = \C^2$.  Writing $w^3 := w \otimes w \otimes w$, the {\sl rational normal curve} $\mathbf{Z} \subset \BP V$ is the curve with affine cone
 \begin{align}
 \widehat{\mathbf{Z}} := \{ w^3 \in \Sym^3 W \mid  w \in W\}.
 \end{align}
 Via the natural $\fgl(W)$-representation on $\Sym^3 W$, $\bZ$ is homogeneous with
 \begin{align}
 \fgl(W) \cong \finf{\mathbf{Z}} \subset \fcsp(V) \subset \fgl(V).
 \end{align}

 Fix a basis $\{ x, y \}$ of $W$, so $\{ x^3, 3x^2y, 3xy^2, y^3 \}$ is a basis of $V$, where $x^2 y := \frac{1}{3} (x \otimes x \otimes y + x \otimes y \otimes x + y \otimes x \otimes x)$ and similarly for $xy^2$. Letting $\{ \theta^1,...,\theta^4 \}$ be the dual basis, there is a symplectic form
 \begin{align}
 \sigma = \theta^1 \wedge \theta^4 - 3 \theta^2 \wedge \theta^3 \in \wedge^2 V^*,
 \end{align}
 unique up to a nonzero scalar multiple, with respect to which $\mathbf{Z}$ is a (nondegenerate) Legendrian curve.
 In particular, $T^{(3)}_z \wbZ = V$ at any $z \in \wbZ \backslash 0$.  (By $\GL(W)$-invariance of $\bZ$, it suffices to verify this assertion at $z = x^3$.)
 \end{example}

\subsection{Curves in $\bbP^3$ and 4th order ODE}
\label{S:4ODE}

 Given any nondegenerate curve $Z \subset \bbP V$, we may consider its corresponding equivalence class under the action of $\GL(V)$, and invariants under this projective action.  The local study of the projective geometry of curves is a classical subject \cite{Lie, Hal1880, Wil1905}, and we summarize here aspects relevant for our study.  To any (unparametrized) nondegenerate curve in $\bbP V \cong \bbP^3$, viewed up to projective transformations, there is an associated linear homogeneous 4th order ODE:
 \begin{align} \label{E:4ODE}
 \cE: \quad u^{(4)} + p_3(t) u''' + p_2(t) u'' + p_1(t) u' + p_0(t) u = 0,
 \end{align}
 viewed up to the most general (point) transformations preserving this class of linear ODE:
 \begin{align} \label{E:ODEeq}
 (\widetilde{t},\widetilde{u}) = (\lambda(t), \mu(t) u).
 \end{align}

 Specifically, suppose that $Z$ is locally expressed in terms of a parameter $t$ via homogeneous coordinates as $[u_0(t): u_1(t): u_2(t): u_3(t)]$, for some functions $u_i(t)$ that are well-defined only up to multiplication by a nonvanishing function $\mu(t)$.  Nondegeneracy implies that $\{ u_i(t) \}$ are linearly independent and there is a unique associated ODE $\cE_Z$ \eqref{E:4ODE} having these as fundamental solutions.  Reparametrizing the curve via $\widetilde{t} = \lambda(t)$, or changing to $\widetilde{u}_i = \mu(t) u_i$ yields an equivalent ODE under \eqref{E:ODEeq}.  Conversely, given \eqref{E:4ODE}, let $\{ u_i(t) \}$ be a fundamental set of solutions, so there is a well-defined local curve $Z_\cE = \{ [u_0(t): u_1(t): u_2(t): u_3(t)] \} \subset \bbP^3$ that is nondegenerate.  Since $\{ u_i(t) \}$ are defined up to invertible linear transformations, then $Z_\cE$ is defined up to projective transformations.

 \begin{definition}
 A {\sl relative invariant of weight $k$} of \eqref{E:4ODE} is a function $I$ of the coefficients $p_i$ such that under \eqref{E:ODEeq}, it is transformed via $\widetilde{I} = \frac{1}{(\lambda')^k} I$.  A relative invariant of weight $0$ is an {\sl absolute invariant}.
 \end{definition}

 Via \eqref{E:ODEeq}, $\cE$ can always be brought to {\sl Laguerre--Forsyth canonical form}:
 \begin{align} \label{E:4ODELF}
 \mbox{LF}: \qquad u^{(4)} + q_1(t) u' + q_0(t) u = 0.
 \end{align}
 (We have dropped tildes here for convenience.)  The residual transformations preserving the Laguerre--Forsyth form \eqref{E:4ODELF} are
 \begin{align} \label{E:ODELFeq}
 (\widetilde{t},\widetilde{u}) = \left( \frac{at+b}{ct+d}, \frac{e}{(ct+d)^3} u\right),
 \end{align}
 where $a,b,c,d,e \in \C$.  Relative invariants of \eqref{E:4ODELF} under \eqref{E:ODELFeq} are similarly defined, and were thoroughly investigated by Wilczynski \cite{Wil1905}.  Of immediate importance to us are the {\sl Wilczynski invariants} $\Theta_3$ and $\Theta_4$, which are relative invariants of weight 3 and 4 respectively.  Expressions are in \cite{Wil1905}, or for example \cite[(2.3)]{Dou2008}.  The function $\Theta_3$ is a (nonzero) constant multiple of $q_1$.

 Restricting now to Legendrian curves $Z \subset \bbP V$, the associated ODE has $\Theta_3 \equiv 0$ \cite{Wil1905, DZ2013}, and so $\cE_Z$ is of the form
 \begin{align} \label{E:LFcan}
 \cE_Z: \quad u^{(4)} + q_0(t) u = 0.
 \end{align}

 \begin{proposition} \label{P:relinv}
 Up to \eqref{E:ODELFeq}, the ODE \eqref{E:LFcan} has relative invariants:
 \begin{enumerate}
 \item[(i)] $q_0$ of weight 4, and
 \item[(ii)] $\cR := 8 q_0 q_0'' - 9 (q_0')^2$ of weight 10.
 \item[(iii)] $\cI := \frac{\cR^2}{4096(q_0)^5}$ of weight 0, i.e.\ it is an absolute invariant.
 \end{enumerate}
 \end{proposition}

 \begin{proof}
 The ODE \eqref{E:LFcan} has $\Theta_3 \equiv 0$.  In this case,
 \begin{enumerate}
 \item[(i)] From \cite[(2.3)]{Dou2008}, $\Theta_4$ is a constant multiple of $q_0$, so $q_0$ is a relative invariant of weight 4.
 \item[(ii)] Wilczynski found another relative invariant \cite[(15) on p.242]{Wil1905}, denoted $\Theta_{4\cdot 1}$ there, which is a constant multiple of $\cR := 8 q_0 q_0'' - 9 (q_0')^2$.  Its weight was not stated in \cite{Wil1905}, so we establish this here via the chain rule. Let $T = \lambda(t) = \frac{at+b}{ct+d}$, with $t$-derivative $\lambda'(t) = \frac{ad-bc}{(ct+d)^2}$.  Write $Q_0(T) = \widetilde{q_0} = \frac{q_0(t)}{(\lambda')^4}$, and denote $T$-derivatives by dots.  Then:
 \begin{align}
 \dot{Q}_0 &= \frac{dQ_0}{dT} = \frac{ \left( \frac{q_0}{(\lambda')^4} \right)'}{\lambda'} = \frac{q_0'}{(\lambda')^5} - \frac{4q_0 \lambda''}{(\lambda')^6} \\
 &= \frac{(ct+d)^9}{(ad-bc)^5} ((ct+d) q_0' + 8 q_0 c) \nonumber\\
 \ddot{Q}_0 &= \frac{d}{dT} \left(  \frac{dQ_0}{dT} \right) = \frac{1}{\lambda'} \left(  \frac{dQ_0}{dT} \right)' \\
 &= \frac{(ct + d)^{10}}{(a d - b c)^6} ((c t + d)^2 q_0'' + 18 c (ct + d) q_0' + 72 q_0 c^2) \nonumber\\
 \widetilde\cR &= 8 Q_0 \ddot{Q}_0 - 9 (\dot{Q}_0)^2 = \frac{(ct + d)^{20}}{(a d - b c)^{10}} (8 q_0 q_0'' - 9 (q_0')^2) = \frac{\cR}{(\lambda')^{10}}
 \end{align}
 \end{enumerate}
 Finally, (iii) follows from (i) and (ii).
 \end{proof}

 See also Appendix \ref{S:R} for how (ii) is easily confirmed in Maple.

 \subsection{Homogeneous nondegenerate Legendrian curves in $\bbP^3$}
 \label{S:HNLC}

 Fix an isomorphism $V \cong \C^4$.  A {\em homogeneous} nondegenerate Legendrian curve $Z \subset \bbP V$ is locally the (image of the) projectivization of a (parametrized) curve in $V \backslash 0$:
 \begin{align} \label{E:wZ}
 \gamma(t) = \exp(tA) z,
 \end{align}
 where $A \in \fcsp(4)$ and $0 \neq z \in \C^4$.  Since the identity matrix acts trivially on $Z$, then we may assume that $A \in \fsp(4)$.  Given such an $A$, we say that a base point $z$ is {\sl $A$-admissible} if \eqref{E:wZ} yields a nondegenerate Legendrian curve.  By nondegeneracy, the osculating sequence at $[\gamma(t)]$ generates $V$, so
 \begin{align} \label{E:oscA}
 \exp(tA) z, \quad \exp(tA) Az, \quad \exp(tA) A^2 z, \quad \exp(tA) A^3 z
 \end{align}
 are linearly independent.  Applying $\exp(-tA)$, these remain linearly independent, so this implies that:\\

 $(\star)$: \quad {\em the minimal and characteristic polynomials of $A$ agree}.\\

 Our aim here is to describe the projective equivalence classes of germs of such curves, i.e.\ under $Z \mapsto PZ$, where $P \in \GL(4)$ is arbitrary.  Two natural related notions of equivalence on $(A,z)$ arise:
 \begin{enumerate}
 \item Given $P \in \GL(4)$, we have a curve $P\gamma(t) = \exp(t \widetilde{A}) \widetilde{z}$, where
 \begin{align} \label{E:equiv1}
 \widetilde{A} = PAP^{-1}, \quad \widetilde{z} = Pz.
 \end{align}
 (The symplectic form is conjugated to another symplectic form.)
 \item Despite introducing a parametrization in \eqref{E:wZ}, we are interested in {\em unparametrized} curves.  Under an affine reparametrization $\widetilde{t} = a t + b$, with $a \in \C^\times$ and $b \in \C$, we have $\exp(\widetilde{t} A) z = \exp(t\widetilde{A}) \widetilde{z}$, where
 \begin{align} \label{E:equiv2}
 \widetilde{A} = a A, \quad \widetilde{z} = \exp(b A) z.
 \end{align}
 \end{enumerate}
 However, such equivalences on $(A,z)$ are insufficient for studying projective equivalence of $Z \subset \bbP V$ associated to \eqref{E:wZ}, as the following example shows.

 \begin{example} \label{X:RNC}
 Consider $(A,z)$ given in (i) and (ii) below, which are clearly inequivalent under the group generated by \eqref{E:equiv1} and \eqref{E:equiv2}.  Our first claim is that (i) and (ii) generate Legendrian curves $Z \subset \bbP^3 = \bbP(\bbC^4)$ via \eqref{E:wZ} that are projectively equivalent to a germ of the rational normal curve $\bZ$.  The last column provides the explicit equivalence.
 \[
 \begin{array}{|c|c|c|c|c|c|} \hline
 & A \in \fsp(4,\sigma) & z & \sigma & \begin{tabular}{c} Image of $(e_1,e_2,e_3,e_4)$ under \\ isomorphism $\C^4 \to \Sym^3 \C^2$ \end{tabular}\\ \hline\hline
 \mbox{(i)} & \begin{psm}
 0 & 0 & 0 & 0\\
 1 & 0 & 0 & 0\\
 0 & 1 & 0 & 0\\
 0 & 0 & 1 & 0
 \end{psm} & \begin{psm}1\\0\\0\\0 \end{psm} &
 \begin{psm}
 0 & 0 & 0 & 1\\
 0 & 0 & -1 & 0\\
 0 & 1 & 0 & 0\\
 -1 & 0 & 0 & 0
 \end{psm} &
 (x^3,3x^2y,6xy^2,6y^3)\\ \hline
  \mbox{(ii)} & \begin{psm}
 3 & 0 & 0 & 0\\
 0 & 1 & 0 & 0\\
 0 & 0 & -1 & 0\\
 0 & 0 & 0 & -3
 \end{psm} & \begin{psm}1\\1\\1\\1 \end{psm} &
 \begin{psm}
 0 & 0 & 0 & 1\\
 0 & 0 & -3 & 0\\
 0 & 3 & 0 & 0\\
 -1 & 0 & 0 & 0
 \end{psm} &
 (x^3,3x^2y,3xy^2,y^3)\\ \hline
 \end{array}
 \]
 Let us provide some more details for (ii).  Note that $A^\top \sigma + \sigma A = 0$ yields $\sigma = \begin{psm}
 0 & 0 & 0 & a\\
 0 & 0 & b & 0\\
 0 & -b & 0 & 0\\
 -a & 0 & 0 & 0
 \end{psm}$.  From $\gamma(t) = \exp(At) z = \begin{psm}
 e^{3t} \\
 e^t\\
 e^{-t}\\
 e^{-3t}
 \end{psm}$, the Legendrian condition $\sigma(\gamma(t),\gamma'(t)) = 0$ forces $b = -3a$.  Via the stated isomorphism, we have
 \begin{align}
 \gamma(t) = e^{3t} x^3 + 3e^t x^2 y + 3e^{-t} xy^2 + e^{-3t} y^3 = (e^t x + e^{-t} y)^3.
 \end{align}
 Thus, this corresponds to $\bZ$.  In (i), $\gamma(t) = \begin{psm} 1\\ t\\ \frac{t^2}{2} \\ \frac{t^3}{6} \end{psm}$, and
details are similar.

 Our second claim is that for these $A$, if $z$ and $\widetilde{z}$ are $A$-admissible base points, then $(A,z)$ and $(A,\widetilde{z})$ are equivalent under \eqref{E:equiv1} via a matrix $P$ commuting with $A$.  Let $\widetilde{z} = \begin{psm} z_1\\z_2\\z_3\\z_4 \end{psm}$ and let $\widetilde\gamma(t) = \exp(At) \widetilde{z}$.  We have:
 \begin{align}
 \begin{array}{|c|c|c|c|} \hline
 & \widetilde\gamma(t) & \begin{tabular}{c} Ndg\\ condition \end{tabular} & P \\ \hline\hline
 \mbox{(i)} & \begin{psm}
 z_1\\
 z_2 + t z_1\\
 z_3 + t z_2 + \frac{t^2}{2} z_1\\
 z_4 + t z_3 + \frac{t^2}{2} z_2 + \frac{t^3}{6} z_1
 \end{psm} & z_1 \neq 0 &
 \begin{psm}
 z_1 & 0 & 0 & 0\\
 z_2 & z_1 & 0 & 0 \\
 z_3 & z_2 & z_1 & 0\\
 z_4 & z_3 & z_2 & z_1
 \end{psm}\\ \hline
 \mbox{(ii)} & \begin{psm}
 e^{3t} z_1 \\
 e^t z_2\\
 e^{-t} z_3\\
 e^{-3t} z_4
 \end{psm} & z_1 z_2 z_3 z_4 \neq 0 & \begin{psm}
 z_1 & 0 & 0 & 0\\
 0 & z_2 & 0 & 0\\
 0 & 0 & z_3 & 0\\
 0 & 0 & 0 & z_4
 \end{psm}\\ \hline
 \end{array}
 \end{align}
 We confirm that $PAP^{-1} = A$ and $\widetilde{z} = Pz$, where $z$ was specified in the previous table.
 \end{example}

 Let us apply Wilczynski theory from Section \ref{S:4ODE}.  First, recall that the spectrum of any $A \in \fsp(4)$ is invariant under negation.  (Any matrix $A$ acting on $V \cong \bbC^4$ induces an action on $V^* \cong \bbC^4$ via $-A^\top$, and existence of an $\fsp(4)$-invariant bilinear form implies that $V \cong V^*$ as $\fsp(4)$-reps.)  By \eqref{E:equiv1}, it suffices to consider the Jordan forms of $A$.  By $(\star)$, these are:
 \begin{align} \label{E:JFA}
 \underset{(\alpha\beta\neq 0,\, \alpha^2 \neq \beta^2)}{\begin{psm}
 \alpha & 0 & 0 & 0\\
 0 & \beta & 0 & 0\\
 0 & 0 & -\beta & 0\\
 0 & 0 & 0 & -\alpha
 \end{psm}}, \quad
 \underset{(\alpha\neq 0)}{\begin{psm}
 \alpha & 1 & 0 & 0\\
 0 & \alpha & 0 & 0\\
 0 & 0 & -\alpha & 1\\
 0 & 0 & 0 & -\alpha
 \end{psm}}, \quad
 \underset{(\alpha\neq 0)}{\begin{psm}
 \alpha & 0 & 0 & 0\\
 0 & 0 & 1 & 0\\
 0 & 0 & 0 & 0\\
 0 & 0 & 0 & -\alpha
 \end{psm}}, \quad
 \begin{psm}
 0 & 0 & 0 & 0\\
 1 & 0 & 0 & 0\\
 0 & 1 & 0 & 0\\
 0 & 0 & 1 & 0
 \end{psm}.
 \end{align}

 In particular, the minimal / characteristic polynomial of $A$ is:
 \begin{align} \label{E:charpoly}
 f_A(s) = s^4 - (\alpha^2+\beta^2) s^2 + \alpha^2 \beta^2,
 \end{align}
 which has roots $\pm \alpha$ and $\pm \beta$ (possibly zero).  By Cayley--Hamilton,  $f_A$ annihilates $A$, so
 \begin{align} \label{E:ACH}
 A^4 - (\alpha^2+\beta^2) A^2 + \alpha^2 \beta^2 \id = 0.
 \end{align}
 Given $Z \subset \bbP V$ determined by $\gamma(t) = \exp(tA) z$, differentiation yields $\gamma^{(i)}(t) = A^i \exp(tA) z$.
 Because of \eqref{E:ACH}, all components of $\gamma$ satisfy the scalar ODE
 \begin{align} \label{E:homODE}
 \cE: \quad u^{(4)} - (\alpha^2+\beta^2) u'' + \alpha^2 \beta^2 u = 0.
 \end{align}
 If $z$ is $A$-admissible, then the components of $\gamma$ are linearly independent, so these are fundamental solutions of \eqref{E:homODE}, and hence $\cE = \cE_Z$.

  If $\alpha^2 + \beta^2 = 0$, then \eqref{E:homODE} is in Laguerre--Forsyth canonical form \eqref{E:LFcan}, we have $q_0 = \alpha^2 \beta^2 = -\alpha^4$.  If $q_0 \neq 0$, then $\cI = 0$ from Proposition \ref{P:relinv}.

  If $\alpha^2+\beta^2 \neq 0$, then \eqref{E:homODE} is brought to canonical form via \eqref{E:ODEeq} with\footnote{See Appendix \ref{S:LF} for details on how to find this transformation.}
 \begin{align} \label{E:LFtransf}
 \lambda = -2\sqrt{10(\alpha^2 + \beta^2)} \tanh\left(t \sqrt{\tfrac{\alpha^2 + \beta^2}{10}} \right), \quad
 \mu = (\lambda')^{3/2}.
 \end{align}
 As a result, we find that \eqref{E:LFcan} has
 \begin{align} \label{E:q0}
 q_0(t) = -\frac{1600 (\alpha^2 - 9 \beta^2)(9 \alpha^2 - \beta^2)}{(t^2 - 40(\alpha^2 + \beta^2))^4}.
 \end{align}
 The absolute invariant $\cI$ from Proposition \ref{P:relinv} is then
 \begin{align} \label{E:absinv1}
 \cI = -\frac{(\alpha^2 + \beta^2)^2}{(\alpha^2 - 9\beta^2)(9\alpha^2 - \beta^2)}.
 \end{align}
 Alternatively, if \eqref{E:charpoly} is written as $f_A(s) = s^4 + c_2 s^2 + c_0$,
 then
 \begin{align} \label{E:CI}
 q_0(t) = -\frac{1600(9(c_2)^2-100 c_0)}{(t^2 + 40c_2)^4}, \quad
 \cI = \frac{(c_2)^2}{ 9(c_2)^2-100 c_0}.
 \end{align}

%

 \begin{theorem} \label{T:homZ}
 Let us denote by $\bL_{r^2}$ with $r \in \bbC$ the projective equivalence class of germs of homogeneous nondegenerate Legendrian curves represented by \eqref{E:wZ} with $z = (1,1,1,1)^\top$ and $A$  as follows.
  \[
 \begin{array}{|c||c|c|c|c|} \hline
 & \bL_{r^2}\,\, (r^2 \neq 0, 1) &  \bL_1 &  \bL_0\\ \hline\hline
 A &  \begin{psm}
 r &  &  & \\
  & 1 &  & \\
  &  & -1 & \\
  &  &  & -r
 \end{psm} &
 \begin{psm}
 1 & 1 &  & \\
 0 & 1 &  & \\
  &  & -1 & 1\\
  &  & 0 & -1
 \end{psm} &
 \begin{psm}
 1 &  &  & \\
  & 0 & 1 & \\
  & 0 & 0 & \\
  &  &  & -1
 \end{psm} \\ \hline
 \end{array} \] Then $\bL_{r^2} = \bL_{1/r^2}$ for $ r \neq 0$ and the family $\{ \bL_{r^2} \mid r \in \C\}$ covers all projective equivalence classes of germs of homogeneous nondegenerate Legendrian curves $Z \subset \bbP^3.$ Furthermore, we can list
   $\dim \finf{Z}$, the relative invariant $q_0$ and the absolute invariant $\cI$   as follows.
 \[
 \begin{array}{|c||c|c|c|c|c|} \hline
 & \bL_9 \cong \bL_{1/9} & \bL_{r^2} \cong \bL_{1/r^2} &  \bL_1 &  \bL_0\\ \hline\hline
 \multirow{2}{*}{$A$} & \begin{psm}
 3 &  &  & \\
  & 1 &  & \\
  &  & -1 & \\
  &  &  & -3
 \end{psm} & \begin{psm}
 r &  &  & \\
  & 1 &  & \\
  &  & -1 & \\
  &  &  & -r
 \end{psm} &
 \begin{psm}
 1 & 1 &  & \\
 0 & 1 &  & \\
  &  & -1 & 1\\
  &  & 0 & -1
 \end{psm} &
 \begin{psm}
 1 &  &  & \\
  & 0 & 1 & \\
  & 0 & 0 & \\
  &  &  & -1
 \end{psm} \\
 & & \mbox{{\tiny $(r^2 \in \C \backslash \{ 0, \tfrac{1}{9}, 1, 9 \})$}} & & \\ \hline
 \dim \finf{Z} & 4 & 2 & 2 & 2\\ \hline
 q_0 & 0 & \neq 0 & \neq 0 & \neq 0 \\
 \cI & - & \tfrac{(r^2+1)^2}{(r^2-9)(9r^2-1)} & -\tfrac{1}{16} & \tfrac{1}{9}\\ \hline
 \end{array}
 \] In particular,  $\bL_{r^2} = \bL_{\widetilde{r}^2}$  for $r^2 \neq \widetilde{r}^2 \in \C$  if and only if $r^2 \widetilde{r}^2 =1 $.
 \end{theorem}

 \begin{proof} Locally, the curve is the projectivization of \eqref{E:wZ}, where without loss of generality $A \in \fsp(4)$  is one of the matrices in \eqref{E:JFA}.  By \eqref{E:equiv2}, we may further rescale these, yielding the matrices in the table above.  We exclude (i) in Example \ref{X:RNC} since (ii) there leads to an equivalent curve $\bZ$ (with $\dim \finf{\bZ} = 4$).

 Our next claim is that any $A$-admissible base point $\widetilde{z} = (z_1,z_2,z_3,z_4)^\top$ can be normalized to $z = (1,1,1,1)^\top$ via a matrix $P$ commuting with $A$. The case $\diag(r,1,-1,-r)$ proceeds just as in Example \ref{X:RNC}.
 \begin{align}
 \begin{array}{|c|c|c|c|} \hline
 A & \widetilde\gamma(t) = \exp(At) \widetilde{z} & \begin{tabular}{c} Ndg\\ condition \end{tabular} & P \\ \hline\hline
 \begin{psm}
 1 & 1 &  & \\
 0 & 1 &  & \\
  &  & -1 & 1\\
  &  & 0 & -1
 \end{psm} & \begin{psm}
 e^t(z_1 + t z_2)\\
 e^t z_2\\
 e^{-t}(z_3 + t z_4)\\
 e^{-t} z_4
 \end{psm} & z_2 z_4 \neq 0 &
 \begin{psm}
 z_2 & z_1 - z_2 & 0 & 0\\
 0 & z_2 & 0 & 0\\
 0 & 0 & z_4 & z_3 - z_4\\
 0 & 0 & 0 & z_4
 \end{psm}\\ \hline
 \begin{psm}
 1 &  &  & \\
  & 0 & 1 & \\
  & 0 & 0 & \\
  &  &  & -1
 \end{psm} & \begin{psm}
 e^t z_1 \\
 z_2 + t z_3\\
 z_3\\
 e^{-t} z_4
 \end{psm} & z_1 z_3 z_4 \neq 0 &
 \begin{psm}
 z_1 & 0 & 0 & 0\\
 0 & z_3 & z_2 - z_3 & 0\\
 0 & 0 & z_3 & 0\\
 0 & 0 & 0 & z_4
 \end{psm}\\ \hline
 \end{array}
 \end{align}

Thus, take $z = (1,1,1,1)^\top$ and $\gamma(t) = \exp(At) z$.  The compatible $\sigma$ are:
 \begin{align}
 \begin{array}{|c||c|c|c|c|} \hline
 & \bL_{r^2} &  \bL_1 &  \bL_0\\ \hline\hline
 \sigma & \begin{psm}
 0 & 0 & 0 & 1\\
 0 & 0 & -r & 0\\
 0 & r & 0 & 0\\
 -1 & 0 & 0 & 0
 \end{psm} &
 \begin{psm}
 0 & 0 & 0 & 1\\
 0 & 0 & -1 & -1\\
 0 & 1 & 0 & 0\\
 -1 & 1 & 0 & 0
 \end{psm} &
 \begin{psm}
 0 & 0 & 0 & 1\\
 0 & 0 & -2 & 0\\
 0 & 2 & 0 & 0\\
 -1 & 0 & 0 & 0
 \end{psm} \\
 & \mbox{{\tiny $(r^2 \in \C \backslash \{ 0, 1 \})$}} & & \\ \hline
 \end{array}
 \end{align}
 From \eqref{E:q0}, $q_0 \equiv 0$ only for $\bL_9 \cong \bL_{1/9}$.  Otherwise, $q_0 \not\equiv 0$, and using the absolute invariant $\cI$ from \eqref{E:absinv1}, we see that $\bL_a \cong \bL_b$ if and only if $b = \frac{1}{a}$.  When $q_0 \not\equiv 0$, a tedious direct calculation verifies that $\finf{Z} = \langle \id, A \rangle$.
 \end{proof}

 \begin{remark}
 When $q_0 \neq 0$, the projective equivalence class $\bL_a$ could have been equivalently labelled using $\cI$, and any $\cI = c\in \bbC$ defines such a class.
 \end{remark}

 \begin{remark} In the proof of Theorem \ref{T:homZ}, the verification that $\dim \finf{Z} = 3$ is impossible for (homogeneous) nondegenerate Legendrian curves $Z \subset \bbP^3$ is straightforward, but tedious.  In Proposition \ref{P:LCgap}, we provide a simple conceptual proof of this fact.
 \end{remark}

%
%

 \section{Symmetries of Contact $G_0$-structures}

 \subsection{Contact $G_0$-structures}

 In this section, we review some  results on contact $G_0$-structures on a contact manifold.
Although we need the theory only for 5-dimensional contact manifolds, our discussion is for arbitrary (odd) dimension, as no extra effort is required for higher dimensions.

Let $(X,H)$ be a {\sl (complex) contact manifold}, i.e.\ $X$ is a complex manifold of odd dimension with a contact structure $H \subset TX$ and contact line bundle $L := TX/H$.
Denote by $\omega: \wedge^2 H \to L$ the homomorphism induced by the Lie bracket of local vector fields, equipping $H_x$ with  the  $L$-valued symplectic form  $\omega_x:  \wedge^2 H_x \to L_x$ for each $x \in X$.  Let $(V, \sigma)$ be a symplectic vector space with $\dim V = \dim X -1$.

 \begin{definition} The {\sl Heisenberg algebra} $\heis = \heis_{-2}   \oplus \heis_{-1}$ is the graded Lie algebra with $\heis_{-1} = V$ and $\heis_{-2} = \C$, and the Lie bracket defined by
 \begin{align}
 [v, w] = \sigma(v, w) \in \heis_{-2}, \quad \forall v, w \in \heis_{-1}.
 \end{align}
 Then $\CSp(V)$ can be identified with the group $\Aut_{{\rm gr}}(\heis)$ of graded Lie algebra automorphisms of $\heis$ (so that $c \cdot {\rm Id}_V, 0 \neq c \in \C,$  acts as the scalar multiplication by $c^2$ on $\heis_{-2}$).  The {\sl Heisenberg group} $\Heis$ is the simply connected complex algebraic group whose Lie algebra is $\heis$.

 The {\sl Heisenberg contact structure} $(\Heis,\sH)$ is the (left) $\Heis$-invariant contact structure $\sH \subset T\Heis$ obtained via left translation of the subspace $V=\heis_{-1}  \subset \heis =  T_o \Heis$ at the identity element $o \in \Heis$.
 \end{definition}

 A coordinate description of $(\Heis,\sH)$ in the 5-dimensional case is given in Proposition \ref{p.Heis}.

\begin{definition}\label{d.frame} \quad
  \begin{itemize}
 \item[(1)] A {\sl contact frame} at $x \in X$ is a linear isomorphism $f: V \to H_x$ such that $\omega_x(f(u), f(v)) =0$ for any $u, v \in V$ satisfying $\sigma(u, v) =0$. Let $\Fr_x(X,H)$ denote the set of all contact frames at $x$.  The {\sl contact frame bundle} $\Fr(X,H) := \bigcup_{x \in X} \Fr_x(X,H)$ is a principal $\CSp(V)$-bundle over $X$.
\item[(2)] Given a Lie subalgebra $\fg_0 \subset \fcsp(V)$, let $G_0 \subset \CSp(V)$ be the (not necessarily closed) connected Lie subgroup with Lie algebra $\fg_0$.  A {\sl contact $G_0$-structure} on $X$ is a $G_0$-principal subbundle $\sP \subset \Fr(X,H)$.
\item[(3)] A holomorphic vector field $\vec{v}$ on an open subset $U \subset X$ is a {\sl symmetry} ({\sl infinitesimal automorphism}) of the contact $G_0$-structure $\sP$ if the natural vector field $\vec{v}'$ on $\Fr(X,H)|_U$ induced by $\vec{v}$ is tangent to $\sP|_U$.
\item[(4)] For $x \in X$, we denote by $\aut(\sP)_x$ the Lie algebra of germs of symmetries of $\sP$ in neighborhoods of $x$.  Let $\aut(\sP)_{x}^0 \subset \aut(\sP)_x$ be the subalgebra consisting of $\vec{v} \in \aut(\sP)_x$ with $\vec{v}(x) =0.$
\end{itemize}
\end{definition}

\begin{example}  For $(\Heis,\sH)$ above, the left $\Heis$-action on $\Heis$ lifts to $\Fr(\Heis, \sH)$ and we identify $\Fr_o(\Heis, \sH) = \CSp(V)$.

For a connected subgroup $G_0 \subset \CSp(V) = \Fr_o(\Heis, \sH)$, its left translation gives a contact $G_0$-structure  $\sP^{G_0} \subset \Fr(\Heis, \sH)$ called  the {\sl flat contact $G_0$-structure}.
\end{example}

\begin{definition}\label{d.Heisenberg}
 A contact $G_0$-structure $\sP \subset \Fr(X,H)$ on a contact manifold $(X,H)$ is {\sl locally flat} if each point $x\in X$ admits  an open neighborhood $x \in U \subset X$ and a biholomorphic map $\varphi: U \to \varphi(U) \subset \Heis$ to an open subset in $\Heis$ such that the differential ${\rm d} \varphi: \Fr(X,H)|_U \to \Fr(\Heis, \sH)|_{\varphi(U)}$ sends $\sP|_U$ to the flat contact $G_0$-structure $\sP^{G_0}|_{\varphi(U)}$.
 \end{definition}

\subsection{Tanaka and contact prolongation}
\label{S:Tanaka}

An important algebraic tool in the study of differential geometric structures is Tanaka prolongation \cite{Tan1970}.

 \begin{definition}
 Let $\fm = \fm_{-\mu} \oplus ... \oplus \fm_{-1}$ be a negatively graded Lie algebra of {\sl depth} $\mu$ that is {\sl bracket-generating}, i.e.\ generated by $\fm_{-1}$.  Let $\fder_{{\rm gr}}(\fm)$ be the graded derivation algebra of $\fm$, which is the Lie algebra of the graded automorphism group $\Aut_{{\rm gr}}(\fm)$.  (Note that $\fder_{{\rm gr}}(\fm) \hookrightarrow \fgl(\fm_{-1})$ by the bracket-generating property.)  Let $\fg_0 \subseteq \fder_{{\rm gr}}(\fm)$ be a Lie subalgebra.  The {\sl Tanaka prolongation} of $(\fm,\fg_0)$ is the unique (up to isomorphism) graded Lie algebra $\pr(\fm,\fg_0) = \oplus_{i= -\mu}^{\infty} \pr_i(\fm, \fg_0)$ such that:
 \begin{enumerate}
 \item[(a)] $\pr_{\leq 0}(\fm,\fg_0)$ agrees with $\fm \oplus \fg_0$;
 \item[(b)] if $x \in \pr_i(\fm,\fg_0)$ for $i > 0$ satisfies $[x,\fm_{-1}] = 0$, then $x = 0$;
 \item[(c)] $\pr(\fm,\fg_0)$ is maximal with respect to the above two properties.
 \end{enumerate}
When $\fg_0 = \fder_{{\rm gr}}(\fm),$ we write $\pr(\fm) := \pr(\fm,\fder_{{\rm gr}}(\fm))$.
 \end{definition}

 For $\fm = \heis$, $\dim \pr(\fm)$ is infinite, but we may consider prolongation in the setting of contact $G_0$-structures:

 \begin{definition}  \label{D:CP}
 Given $\fm = \heis$ and $\fg_0 \subsetneq \fder_{{\rm gr}}(\fm)$, define $\fg = \pr(\fm,\fg_0)$.  We say that $A \in \fg_1$ (which is uniquely determined by a map $\heis \to \fg_0$) is a {\sl contact prolongation of $\fg_0$}.  If no nonzero contact prolongations exist, then
 \begin{align}
 \pr(\fm,\fg_0) \cong \fm \oplus \fg_0.
 \end{align}
 \end{definition}

 We now focus on $\fg_0 = \finf{Z} \subset \bbP V$ when $\dim V = 4$.

 \begin{theorem}\label{t.Abel} Let $Z \subset \BP V \cong \bbP^3$ be a nondegenerate Legendrian curve and $\mathbf{Z} \subset \BP \Sym^3 W$ the rational normal curve.  Then $\fg_0 = \finf{Z} \subset \fcsp(V)$ has a nonzero contact prolongation if and only if there exists a symplectic isomorphism $f: V \to \Sym^3 W$ such that $f(\widehat{Z}) \subset \widehat{\mathbf{Z}}$.
 \end{theorem}

The above is \cite[Thm.5]{HAb}.  (There, the submanifold $Z$ is assumed to be an algebraic subvariety of $\BP V$, but this assumption is not used in the proof.) Thus, generally $\fg = \pr(\fm,\fg_0) \cong \fm \oplus \fg_0$ except in the rational normal case $\bZ$.  Here, $\fg_0 = \finf{\mathbf{Z}}\cong \fgl(2)$, and $\fg = \pr(\fm,\fg_0)$ is the 14-dimensional exceptional complex simple Lie algebra $\fg$ of type $G_2$, equipped with a {\sl contact grading} $\fg = \fg_{-2} \oplus ... \oplus \fg_2$.  Let us put this into a broader context in order to introduce other prolongation results relevant to our study.

Recall that gradings on complex semisimple Lie algebras $\fg$ (with choice of Cartan subalgebra $\fh$) are classified by subsets $\mathcal{S}$ of simple roots $\{ \alpha_i \} \subset \fh^*$ (marked by crosses on Dynkin diagrams), with grading induced on the associated root space $\fg_\alpha$ via their $\mathcal{S}$-height, i.e.\ if $\alpha = \sum_i m_i \alpha_i$, then $\operatorname{ht}(\alpha) = \sum_{i \in \mathcal{S}} m_i$.  Since $\fg \cong \fg^*$ via the Killing form,  the grading is symmetric: $\fg = \fg_{-\mu} \oplus ... \oplus \fg_\mu$.  The non-negative part $\fp = \fg_{\geq 0}$ is a parabolic subalgebra.  Declaring $\fm = \fg_-$, we may ask when $\pr(\fm,\fg_0)$ or $\pr(\fm)$ is isomorphic to $\fg$ itself, and this was addressed in \cite{Yam1993} for all simple $\fg$.  In the $G_2$ case, the results are stated in Figure \ref{F:G2pr}.  (In the first case, we note the Heisenberg algebra $\heis \cong \fg_-$ and $\fg_0 \cong \fgl(2)$.)

 \begin{figure}[h]
 \[
 \begin{array}{|@{}c@{}|@{}c@{}|@{}c@{}|}\hline
 \raisebox{0.1in}{\Gdd{wx}{}} & \raisebox{0.1in}{\Gdd{xw}{}} & \hspace{-0.2in} \raisebox{0.1in}{\quad\,\,\Gdd{xx}{}} \\ \hline
 \begin{array}{c} \vspace{-0.17in}
 \begin{tikzpicture}[scale=0.7]
\foreach \i in {-2,...,2}
	\draw[\graycolor, yshift=8.6*\i mm] (180:-2) to (180:2) node[left]{\tiny $\i$};
\foreach \i in {240,300}
	\draw[->] (0,0) to (\i:1);
\foreach \i in {210, 270, 330}
	\draw[->] (0,0) to (\i:{sqrt(3)});
\foreach \i in {0,60,...,180}
	\draw[->, \greencolor] (0,0) to (\i:1);
\foreach \i in {30,90,150}
	\draw[->, \greencolor] (0,0) to (\i:{sqrt(3)});
\fill[\greencolor] (0,0) circle (0.1);
 \node[right] at (1,0) {$\alpha_1$};
 \node[above] at (150:{sqrt(3)}) {$\alpha_2$};
\end{tikzpicture}
\end{array} &
\begin{array}{c}
\begin{tikzpicture}[scale=0.7]
\foreach \i in {-3,...,3}
	\draw[\graycolor, xshift=3*\i mm, yshift=4*\i mm]  (150:-2) to (150:2) node[above left]{\tiny $\i$};
\foreach \i in {180,240,300}
	\draw[->] (0,0) to (\i:1);
\foreach \i in {210, 270}
	\draw[->] (0,0) to (\i:{sqrt(3)});
\foreach \i in {0,60,120}
	\draw[->, \greencolor] (0,0) to (\i:1);
\foreach \i in {-30,30,...,150}
	\draw[->, \greencolor] (0,0) to (\i:{sqrt(3)});
\fill[\greencolor] (0,0) circle (0.1);
 \node[right] at (1,0) {$\alpha_1$};
 \node[above] at (150:{sqrt(3)}) {$\alpha_2$};
\end{tikzpicture}
\end{array} &
\begin{array}{c} \vspace{-0.05in}
\begin{tikzpicture}[scale=0.7]
\foreach \i in {-5,...,5}
	\draw[\graycolor, xshift=\i mm, yshift=3.1*\i mm]  (160.8:-2) to (160.8:2) node[above=1mm, left=0mm]{\tiny $\i$};
\foreach \i in {180,240,300}
	\draw[->] (0,0) to (\i:1);
\foreach \i in {210, 270, 330}
	\draw[->] (0,0) to (\i:{sqrt(3)});
\foreach \i in {0,60,120}
	\draw[->, \greencolor] (0,0) to (\i:1);
\foreach \i in {30,90,150}
	\draw[->, \greencolor] (0,0) to (\i:{sqrt(3)});
\fill[\greencolor] (0,0) circle (0.1);
 \node[right] at (1,0) {$\alpha_1$};
 \node[above] at (150:{sqrt(3)}) {$\alpha_2$};
\end{tikzpicture}
\end{array}\\ \hline
\pr(\fm,\fg_0) \cong \fg &
\pr(\fm) \cong \fg &
\pr(\fm) \cong \fg \\ \hline
 \end{array}
 \]
 \caption{Tanaka prolongations associated to $\fg$ of type $G_2$}
 \label{F:G2pr}
 \end{figure}

 For geometric purposes, it is better to equip $\fg$ with a decreasing filtration $\fg = \fg^{-\mu} \supset ... \supset \fg^\mu$ via $\fg^i := \bigoplus_{i \geq j} \fg_j$, so that $\fp = \fg^0$.  If $G$ is the corresponding connected, simply-connected Lie group, then we let the parabolic subgroup $P \subset G$ be the corresponding connected Lie subgroup with Lie algebra $\fp$.  The three choices of $P$ in Figure \ref{F:G2pr} are denoted by $P_2, P_1, P_{1,2}$ respectively, and there is an associated double fibration (Figure \ref{F:fibration-flat}).

\subsection{Normalization conditions and local flatness}

\begin{definition}\label{d.prolong} Let $(V, \sigma)$ be a symplectic vector space and $\fg_0 \subset \fcsp(V)$.
\begin{itemize}
\item[(1)] Given $A \in \Hom(V, \fg_0)$, write $A_u := A(u) \in \fg_0$. Denote by $\vec{A} \in V$ the unique vector satisfying
 \begin{align*}
 \sigma( \vec{A}, u) = \frac{2}{\dim V} \tr (A_u), \quad \forall u \in V.
 \end{align*}
 \item[(2)] Define a homomorphism $\delta: \Hom( V, \fg_0) \to \Hom (\wedge^2 V, V)$ by
 \begin{align*}
 \delta A (u, v) = A_u(v) - A_v(u)  - \sigma (u, v) \ \vec{A}, \quad \forall  u, v \in V.
 \end{align*}
 \item[(3)] A subspace $W \subset \Hom(\wedge^2 V, V)$ complementary to ${\rm Im}(\delta)$ is called a {\sl normalization condition}.
  \end{itemize}
\end{definition}

\begin{remark}
The subspace $\ker(\delta) \subset \Hom(V,\fg_0)$ can be identified with the space of contact prolongations (denoted $\fg_1$) mentioned in Definition \ref{D:CP}.
\end{remark}

The following is a special case of Theorem 8.3 of \cite{Tan1970}. It can be also deduced from Theorem 4 of \cite{HAb}, which is a reformulation of Tanaka's result in the setting of contact $G_0$-structures.

\begin{theorem}\label{t.Tanaka}
Suppose that $\fg_0 \subset \fcsp(V)$ is a Lie subalgebra with no nonzero contact prolongation, i.e.\ $\pr(\heis,\fg_0) \cong  \heis \oplus \fg_0$, and let $G_0 \subset \CSp(V)$ be the connected Lie subgroup with Lie algebra $\fg_0$.  Fix a normalization condition.  Then for any contact $G_0$-structure $\sP \subset \Fr(X,H)$, there exists a natural absolute parallelism $\theta$ on $\sP$ such that $\theta$ is invariant under $\vec{v}'$ for any symmetry $\vec{v}$ of the contact $G_0$-structure.
\end{theorem}

\begin{corollary}\label{c.Tanaka}
In the setting of Theorem \ref{t.Tanaka}:
 \begin{itemize}
 \item[(i)] the isotropy representation ${\bf jet}^H_x: \aut(\sP)_x^0 \to \End(H_x)$ is injective;
 \item[(ii)]
$\dim \aut(\sP)_x \leq \dim \sP = \dim X + \dim \fg_0$ for any $x \in X$; and
\item[(iii)]  if $\dim \aut(\sP)_x = \dim \sP$, then $(\sP, \theta)$ is locally isomorphic (near $x$) to the Maurer--Cartan form on a Lie group with Lie algebra $\aut(\sP)_x$. \end{itemize} \end{corollary}

\begin{proof}
By Theorem \ref{t.Tanaka}, a symmetry $\vec{v}$ of the contact $G$-structure $\sP$ on an open subset $U \subset X$ can be lifted to a vector field $\vec{v}'$ on $\sP|_U$ that is a symmetry of $\theta$.  Thus, if ${\bf jet}^H_x(\vec{v})=0$, then $\vec{v}'$ on a neighborhood of $\sP_x$ in $\sP$ must vanish at all points in $\sP_x$. But $\vec{v}'$ is a symmetry of the absolute parallelism $\theta$, so if it vanishes at a point, then it should vanish identically, so (i) follows.
The image of ${\bf jet}^H_x$ must be contained in a subalgebra of $\End(H_x)$ isomorphic to $\fg_0$, so
by (i), we have $\dim \aut(\sP)_x \leq \dim X + \dim \fg_0 = \dim \sP,$ which proves (ii).  Finally, \cite[Thm.8.16]{Olv} implies (iii).
 \end{proof}

A priori, the absolute parallelism $\theta$ in Theorem \ref{t.Tanaka} may not be a Cartan connection on $\sP$ because the normalization condition $W \subset \Hom(\wedge^2 V, V)$ may not be $G_0$-invariant. The obstructions to constructing a natural Cartan connection on $\sP$ and the curvatures of the natural Cartan connection when it exists have been explicitly calculated in \cite{HLSG} and \cite{HM}.  Here is a summary:

\begin{theorem}\label{t.HL}
 In the setting of Theorem \ref{t.Tanaka}, let $\fg = \pr(\heis,\fg_0) =  \heis \oplus \fg_0$.
 Let $\Hom(\wedge^2 \heis, \fg)_{\ell} \subset \Hom(\wedge^2 \heis, \fg)$ be the grade $\ell$ subspace.
Associated with a contact $G_0$-structure $\sP \subset \Fr(X,H)$, there exists a holomorphic section $\tau_{\ell}$ of  the vector bundle $\sP \times_{G} \Hom(\wedge^2 \heis,\fg)_{\ell}$ on $X$ for each $\ell \geq 1$ such that
 \begin{itemize}
 \item[(a)] $\tau_1$ is canonically determined by $\sP$;
 \item[(b)] $\tau_{\ell +1}$ is canonically determined by $\sP$ if $\tau_{1} = \cdots = \tau_{\ell} =0$; and
 \item[(c)] if $\tau_{\ell} =0$ for all $\ell \geq 1$, then the contact $G_0$-structure $\sP$ is locally flat.
 \end{itemize}
 Moreover, (a) implies that $\tau_1$ is annihilated by $\aut({\sP})_x$, and (b) implies that
 $\tau_{\ell+1}$  is annihilated by $\aut({\sP})_x$ if $\tau_1 = \cdots = \tau_{\ell} =0$.
\end{theorem}

 We remark that since $\fg = \fg_{-2} \oplus \fg_{-1} \oplus \fg_0$ and $\dim \fg_{-2} = 1$, then $1 \leq \ell \leq 3$ are the relevant grades for $\Hom(\wedge^2 \heis, \fg)_{\ell}$ appearing above.

The sections $\tau_{\ell}$ (for $\ell \geq 1$) in Theorem \ref{t.HL} are constructed  in the proof of Theorem 2.17 of \cite{HLSG}  or Proposition 7.2 of \cite{HM}. The property  (c) is a consequence of Theorem 2.6 of \cite{HLG2}, a reformulation of Theorem 2.17 of \cite{HLSG}.
One consequence of these constructions is the following.

 \begin{proposition} \label{p.HL}
In the setting of Corollary \ref{c.Tanaka}, assume that there exists an open subset $U \subset X$ such that for each point $x\in U$, there exists $\vec{v}_x \in \aut(\sP)^0_x$ with $ {\bf jet}^H_x(\vec{v}_x) = c_x \cdot {\rm Id}_{H_x}$ for some $0 \neq c_x \in \C$.
Then  the contact $G_0$-structure $\sP$ is locally flat.
 \end{proposition}

 \begin{proof} Let $x \in U$.  By hypothesis, $\vec{v}_x \in \aut(\sP)^0_x$ acts by $c_x$ on $H_x$ (of degree $-1$), so it acts by $-\ell c_x$ on $\Hom(\wedge^2 \heis, \fg)_{\ell}$ for each $\ell \geq 1$.  Since it annihilates $\tau_{\ell}$ in Theorem \ref{t.HL}, then $\tau_{\ell}(x) =0$.  Theorem \ref{t.HL}(c) now yields the result.
 \end{proof}

%

\section{Symmetries of Legendrian cone structures}\label{s.cone}

We henceforth specialize to a 5-dimensional contact manifold $(X,H)$, and let $\psi: \BP H \to X$ be the associated $\BP^3$-bundle on $X$.  We define the notion of a Legendrian cone structure $(X,H,\sC)$, in particular a $Z$-isotrivial cone structure associated to a Legendrian curve $Z \subset \BP V \cong \BP^3$.

 \begin{definition}\label{d.cone}
 A complex submanifold $\sC \subset \BP H$ of dimension $6$ is a {\sl Legendrian cone structure} on $X$ if:
 \begin{itemize}
 \item[(i)] the restriction $\pi:= \psi|_{\sC}: \sC \to X$ is submersive at each point of $\sC$;
 \item[(ii)] for each $x \in X$, the fiber $\sC_x:= \pi^{-1}(x) $ is a Legendrian curve with respect to the symplectic form $\omega_x$ on $H_x$.
 \end{itemize}
 We denote it by the triple $(X,H,\sC)$ or simply by $\sC$. It is {\sl nondegenerate} if $\sC_x$ is nondegenerate for a general point $x \in X$.
\end{definition}

\begin{definition}\label{d.isotrivial} Let $(X,H,\sC)$ be a Legendrian cone structure.
\begin{itemize}
 \item[(i)] Given a symplectic vector space $(V, \sigma)$ and a Legendrian curve $Z \subset \BP V$, we say that $\sC \subset \BP H$ is
$Z$-{\sl isotrivial} if there exists a nonempty open subset $U \subset X$ and an open subset $\sC^U \subset \sC \cap \pi^{-1}(U)$ such that for each $x \in U$, the fiber $\sC^U_x := \sC^U \cap \pi^{-1}(x) \subset \BP H_x$ is isomorphic to $Z \subset \BP V$ under a symplectic isomorphism $(H_x, \omega_x) \stackrel{\cong}{\longrightarrow} (V, \sigma)$.
\item[(ii)] We say that $\sC \subset \BP H$ is {\sl isotrivial} if it is $Z$-isotrivial for some  Legendrian curve $Z \subset \BP V$. \end{itemize}
\end{definition}

\begin{definition}\label{d.flatcone}
Let $(\Heis, \sH)$ be the Heisenberg contact structure.
\begin{itemize}
\item[(i)] Identifying $\sH_o = V$, a Legendrian curve $Z \subset \BP V$ determines a left-invariant fiber subbundle $\sC^Z \subset \BP \sH$, which is the {\sl flat $Z$-isotrivial Legendrian cone structure} $(\Heis, \sH, \sC^Z)$.
\item[(ii)]
A $Z$-isotrivial Legendrian cone structure $\sC \subset \BP H$ on a contact manifold $(X, H)$ is {\sl locally flat} if there exists an open subset $U \subset X$ and a biholomorphic map $\varphi: U \to U' \subset \Heis$ such that the differential ${\rm d} \varphi: \BP TU \to \BP TU'$ sends $\sC|_U \subset \BP H|_U$ to $\sC^Z|_{U'} \subset \BP \sH|_{U'}.$ \end{itemize} \end{definition}

\begin{definition}\label{d.aut}
Let $\sC \subset \BP H$ be a Legendrian cone structure on a contact manifold $(X, H)$.  For a holomorphic vector field $\vec{v}$  on an open subset $U \subset X$, let $\check{v}$ be its canonical lift to $\BP TX|_U$.
 \begin{itemize}
 \item[(i)] $\vec{v}$ is a {\sl contact vector field} if $\check{v}$ is tangent to $\BP H$.
 \item[(ii)] A contact vector field $\vec{v}$  is an {\sl infinitesimal automorphism} of $\sC \subset \BP H$ if $\check{v}$ is tangent to $\sC \cap \BP H|_U.$
 \item[(iii)] For  $x \in U$, denote by $\aut(\sC)_x$  the Lie algebra of germs of infinitesimal automorphisms of $\sC \subset \BP H $ in some neighborhoods of $x$ and by $\aut(\sC)^0_x$ the subalgebra of elements of $\aut(\sC)_x$ that vanish at $x$.
 \item[(iv)] $\sC$ is {\sl transitive at} $x \in X$ if  $\aut(\sC)_x$ generates $T_x X$. If moreover $\aut(\sC)_x^0 \neq 0$, then $\sC$ is {\sl  multiply-transitive at} $x \in X$.
 \item[(v)] $\sC$ is {\sl transitive} (resp.\ {\sl multiply-transitive}) if it is transitive (resp.\  multiply-transitive) at some (hence a general)  point $x  \in X$. We say that $\sC$ is {\sl nontransitive} if it is not transitive.
\end{itemize} \end{definition}

 Here is an explicit description of $(\Heis,\sH,\sC^Z)$.

 \begin{proposition} \label{p.Heis}
For a germ of Legendrian curves $Z \subset \BP^3$, let $\sC^Z \subset \BP \sH$ be the flat $Z$-isotrivial Legendrian cone structure on the 5-dimensional Heisenberg group $\Heis$ from Definition \ref{d.flatcone}.   Let $o \in \Heis$ be the identity element and let $\sC^Z_o \subset \BP \sH_o$ be the fiber of $\sC^Z$ at $o$.    Then there is a linear coordinate system $(x_1, \ldots, x_5)$ on $\Heis$ with the following properties.  \begin{itemize}
 \item[(i)] the Heisenberg group multiplication $(x_1, \ldots, x_5) \circ (\widetilde{x}_1, \ldots, \widetilde{x}_5) $ of two points $(x_1, \ldots, x_5), (\widetilde{x}_1, \ldots, \widetilde{x}_5)$ is $$(x_1+ \widetilde{x}_1, \ldots, x_4 + \widetilde{x}_4, x_5 + \widetilde{x}_5 + \frac{x_1\widetilde{x}_3 - x_3 \widetilde{x}_1 + x_2 \widetilde{x}_4 - x_4 \widetilde{x}_2}{2}).$$
 \item[(ii)] The contact structure $\sH$ is given by the left-invariant form
 \[
 x_1 {\rm d} x_3 - x_3 {\rm d} x_1 + x_2 {\rm d} x_4 - x_4 {\rm d} x_2 - 2 \ {\rm d} x_5.
 \]
 \item[(iii)] The submanifold $\sC^{Z} $ is precisely tangent directions of affine lines on $\Heis$ in the coordinates $(x_1, \ldots, x_5)$ that are  left translates by $\Heis$ of the affine lines through $o$ in the direction of $\sC^Z_o \subset \BP \sH_o :$
 \[
 \{ (t w_1, \ldots, t w_4, 0) \mid t \in \C, [w_1, \ldots, w_4] \in \sC^Z_o \subset \BP \sH_o\}.
 \]
 \item[(iv)] The local automorphism group of the cone structure $\sC^{Z}$ includes the left translation by $\Heis$ as well as the weighted $\C^\times$-action on $\Heis$ given by $$(x_1, \ldots, x_4, x_5) \mapsto (t x_1, \ldots, t x_4, t^2 x_5) \mbox{ for } t \in \C^{\times}.$$ \end{itemize}
 \end{proposition}

 \begin{proof}
 (i) and (ii) are from Lemma 5.2 of \cite{HMSJ},  (iii) is from Definitions 4.2 and 4.4 of \cite{HMSJ}, and (iv) is obvious from (iii).
 \end{proof}

We omit the easy proof of the following Lemma.

\begin{lemma}\label{l.Z-iso}
 Let  $Z \subset \BP V \cong \BP^3$ be a Legendrian curve and let $G_0$ be the identity component of $ \Aut(Z) \subset \CSp(V)$.
 \begin{itemize}
 \item[(i)] Let $(X,H,\sC)$ be a $Z$-isotrivial Legendrian cone structure. For each $x \in X$, define $\sP_x \subset \Fr_x(X, H)$ as a connected component of the set of contact frames $f: V \to H_x$ that send $Z \subset \BP V $ to $\sC_x \subset \BP H_x$. Then $\sP = \bigcup_{x \in X} \sP_x$ is a contact $G_0$-structure on $(X, H)$.
\item[(ii)] Let $\sP \subset \Fr(X, H)$ be a contact $G_0$-structure. For each $x \in X$, define $\sC_x \subset \BP H_x$ as the image $f(Z)$ for any $f \in \sP_x$, which does not depend on the choice of $f \in \sP_x$. Then $\sC = \bigcup_{x \in X} \sC_x \subset \BP H$ is a $Z$-isotrivial Legendrian cone structure on $X$.
    \item[(iii)] In (i) or (ii), we have $\aut(\sC)_x = \aut(\sP)_x$ for any $x \in X$.
    \item[(iv)] In (i) or (ii), the contact $G_0$-structure $\sP$ is locally flat if and only if the $Z$-isotrivial Legendrian cone structure $\sC$ is locally flat. \end{itemize} \end{lemma}

 \begin{proposition}\label{p.flat}
Let $(X,H,\sC)$ be a $Z$-isotrivial Legendrian cone structure, where $Z \subset \BP V$ is a nondegenerate Legendrian curve whose germ is distinct from a germ of the rational normal curve $\mathbf{Z} \subset \BP \Sym^3 W$. Then:
 \begin{itemize}
 \item[(i)] $\dim \aut(\sC)_x \leq 5 + \dim \finf{Z}$ for any $x \in X$.
 \item[(ii)] $\sC$ is locally flat if and only if $\dim \aut(\sC)_x = 5 + \dim \finf{Z}$ for any $x \in X$.
 \item[(iii)] Suppose that $Z \subset \BP V$ is not homogeneous and $0\neq \aut(\sC)_x^0$ for a general $x \in X$.  Then $\sC \subset \BP H$ is locally flat and $\aut(\sC)_x \cong \heis \rtimes \finf{Z}$ for a general $x \in X$.
 \item[(iv)] Suppose that $Z \subset \BP V$ is homogeneous. If
 \begin{align*}
 0 < \dim \aut(\sC)^0_x < \dim \finf{Z}
 \end{align*} for a general $x \in X$, then there is an open subset $U \subset X$ such that  the image of $\mathbf{jet}_x^H$ acts nontrivially on $Z$ for  any $x  \in U$.
 \end{itemize}
 \end{proposition}

\begin{proof}
By Theorem \ref{t.Abel} and Lemma \ref{l.Z-iso}, (i) is a direct consequence of Corollary \ref{c.Tanaka}(i).  In (ii), the forward direction is immediate, so assume that $\dim \aut(\sC)_x = 5 + \dim \finf{Z}$.  By Corollary \ref{c.Tanaka}, we have $\dim \aut(\sP)_y= \dim \sP$ for all $y$ in an open subset $U \subset X$, and the injective homomorphism $\mathbf{jet}^H_x: \aut(\sP)^0_x \to \End(H_x)$ has image $\finf{\sC_x}$, which contains $\C\, {\rm Id}_{H_x}$. Thus, the condition for Proposition \ref{p.HL} is satisfied and $\sC$ is locally flat, so (ii) is proven.  Proposition \ref{p.HL} similarly establishes (iii).

To prove (iv),  assume the contrary that the image of $\mathbf{jet}_x^H: \aut(\sP)^0_x \to \End(H_x)$ acts trivially on $\sC_x$ for all $x$ in a dense subset of $X$. Then by continuity, it acts trivially on $\sC_x$ for all $x\in X$.  Since $Z$ is nondegenerate in $\BP V$, then the image of $\mathbf{jet}_x^H$ is contained in $\C\, {\rm Id}_V \subset \finf{Z}$ for all $x \in X$. Thus by Proposition \ref{p.HL} and Lemma \ref{l.Z-iso}, the Legendrian cone structure must be locally flat. Then $\dim (\sC)^0_x = \dim \finf{Z}$ by (ii), a contradiction to the assumption $\dim \aut(\sC)^0_x < \dim \finf{Z}$.
\end{proof}

\section{Legendrian cone structures of VMRT-type and the canonical double fibration} \label{S:VMRT}

\begin{definition}\label{d.line}
Let $X$ be a complex manifold of dimension 5 with a contact structure $H \subset TX$.
A nonsingular rational curve $\BP^1 \cong C \subset X$ is a {\sl line} if its normal bundle $N_C$ is isomorphic to  $\sO(1) \oplus \sO^{\oplus 3}$. It is easy to see that lines are tangent to $H$. \end{definition}

\begin{lemma}\label{l.line}
For  a complex manifold $X$ of dimension $5$ with a contact structure $H \subset TX$,
let $\sL$ be the quotient line bundle $TX/H$.
A smooth rational curve $C \subset X$ is a line if and only if
\begin{itemize} \item[(i)] deformations of $C$ cover an open subset in $X$; and \item[(ii)] $C$ is of degree 1 with respect to $\sL,$ namely, $\sL|_C \cong \sO(1).$   \end{itemize} Moreover, a line $C \subset X$ is tangent to the contact distribution $H$. \end{lemma}

\begin{proof}
Recall from the deformation theory of rational curves (see Section 1.1 of \cite{HM99}), that deformations of a smooth rational curve $C \subset X$ in a complex manifold covers an open subset if and only if its normal bundle $N_C$ is semipositive, namely, $$N_C \cong \sO(a_1) \oplus \cdots \oplus \sO(a_{\dim X -1})$$ for some nonnegative integers $a_1 \geq \cdots \geq a_{\dim X -1} \geq 0$.
Also recall (for example, from (2.2) of \cite{LB}) that on a complex manifold $X$ of dimension $2m+1$ with a contact structure $H \subset TX$,
the quotient line bundle  $\sL := TX/H$  satisfies ${\rm det}\, TX = \sL^{\otimes (m+1)}.$

  If $C \subset X$ is a line in a $5$-dimensional contact manifold, it satisfies (i) because $N_C$ is semipositive. Moreover,
  \[
  TX|_C \cong TC \oplus N_C \cong \sO(2) \oplus \sO(1) \oplus \sO^{\oplus 3}.
  \]
  Thus  $\det TX|_C = \sO(3) = \sL^{m+1}|_C$ with $m=2$ implies  that $\sL|_C$  is $\sO(1),$  namely, it satisfies (ii).  Conversely, suppose that $C \subset X$ is a nonsingular rational curve satisfying (i) and (ii).  The semipositivity of $N_C$ says  $N_C \cong \sO(a_1) \oplus \cdots \oplus \sO(a_4)$ with $a_1 \geq \cdots \geq a_4 \geq 0$. (ii) implies that $\det N_C \cong \sO(1)$. Consequently, $a_1 =1, a_2=a_3 = a_4 =0$, and $C \subset X$ must be a line in the sense of Definition \ref{d.line}.

For a line $C \subset X$,  the quotient map $$TX|_C = TC \oplus N_C \to (TX/H)|_C = \sL|_C  \cong \sO(1)$$ annihilates $TC  \cong \sO(2)$ because there is no nonzero holomorphic map from $\sO(2)$ to $\sO(1)$. Thus $TC \subset H|_C$, namely,   $C \subset X$ should be tangent to the contact distribution.
\end{proof}

\begin{definition}\label{d.vmrt}
Let $X$ be a complex manifold of dimension 5 with a contact structure $H \subset TX$.
A Legendrian cone structure $\sC \subset \BP H$  is {\sl of VMRT-type} if there exists a 5-dimensional complex manifold $Y$ with a $\BP^1$-fibration $\rho: \sC \to Y$ with the following properties. Let  $\mu: \sC \to X$ be the projection given by $\BP H \to X$.
\begin{itemize} \item[(i)] For each $y \in Y$, the projection $\mu$ sends the fiber $\rho^{-1}(y)$  biholomorphically to a  line $C_y:= \mu(\rho^{-1}(y)) \subset X.$
 \item[(ii)] For each $y \in Y$ and $x \in C_y$, the point $x^{\sharp} \in \rho^{-1}(y) \subset \BP T_x X$ satisfying $\mu(x^{\sharp}) = x$ coincides with the tangent direction $\BP T_x C_y \in \BP T_x X$ of the line $C_y$ at $x$.
    \end{itemize}
    \end{definition}

The following is a reformulation of  Lemma 3.3 from \cite{HLNa}.

    \begin{proposition}\label{p.vmrt}
    Let $(X, H)$ be a contact manifold of dimension $5.$  Let ${\rm Douady}(X)$ be the Douady space of all compact complex submanifolds of $X$. Then the subset ${\rm Lines}(X) \subset {\rm Douady}(X)$ of all lines on $X$ is a nonsingular open subset of ${\rm Douady}(X)$. For each line $C \subset X$,  there exists  an open neighborhood $Y \subset {\rm Lines}(X)$ of the point $[C] \in {\rm Lines}(X)$ such that the universal family  ${\rm Univ}_Y$ with the universal $\BP^1$-bundle map ${\rm Univ}_Y \to Y$ can be embedded as a submanifold $\sC \subset \BP H$ and  the induced morphism $\rho: \sC \to Y$ provides $\sC$ with a Legendrian cone structure of VMRT-type. Conversely, any Legendrian cone structure of VMRT-type arises this way.  \end{proposition}

From Proposition \ref{p.vmrt}, when we consider a Legendrian cone structure of VMRT-type, we identify $\rho: \sC \to Y$ with the universal family over an open subset $Y \subset {\rm Douady}(X)$.

\begin{definition}\label{d.D}
Let $\sC \subset \BP H$ be a Legendrian cone structure of VMRT-type on a contact manifold $(X, H)$ of dimension 5. For each $y \in Y$, the  deformation theory of compact complex submanifold in a complex manifold identifies the tangent space $T_y Y$ with the vector space $H^0(C_y, N_{C_y})$ of holomorphic sections of the holomorphic vector bundle $N_{C_y}$. Let $D_y \subset T_yY$ be the 2-dimensional subspace corresponding to the subspace of holomorphic sections of the vector subbundle $\sO(1) \subset N_{C_y}$, $$D_y \cong H^0(\BP^1, \sO(1)) \subset H^0(\BP^1, \sO(1) \oplus \sO^{\oplus 3}) = H^0(C_y, N_{C_y}) = T_y Y.$$ The rank 2 distribution $D \subset TY$ is called the {\sl natural distribution} on $Y$. \end{definition}

The following proposition is  a direct consequence of Theorem 5.3 of \cite{HLGC} or Theorem 3.15 of \cite{HLNa}.

 \begin{proposition}\label{p.HL235}
Let $\sC \subset \BP H$ with $\rho: \sC \to Y$ be a nondegenerate Legendrian cone structure of VMRT-type.
Let $\sC^o \subset \sC$ be the open subset consisting of nondegenerate points of the Legendrian curves  $\sC_x \subset \BP H_x, x \in X$.  Then there is a dense open subset $Y^o \subset Y$ such that $\sC^o = \rho^{-1}(Y^o)$ and   the restriction $D|_{Y^o}$  of the natural distribution $D$  is a $(2,3,5)$-distribution.
 \end{proposition}

We have the following converse, Theorem 5.10 of \cite{HLGC}.

 \begin{proposition} \label{p.HLconverse}
Let $D \subset TM$ be a $(2,3,5)$-distribution on a 5-dimensional manifold.
Then any point $y \in M$ has a neighborhood $y \in Y \subset M$  such that the $\BP^1$-bundle $\BP D|_Y$ admits a holomorphic submersion $\mu: \BP D|_Y \to X$ to a contact manifold $(X, H)$ of dimension 5 with an embedding $\BP D|_Y \subset \BP H$ whose image $\sC \subset \BP H$ is a nondegenerate Legendrian cone structure of VMRT-type on $(X,H)$  such that any  point $z \in \sC$ is a nondegenerate point of the Legendrian curve $\sC_x, x = \mu(z) \in X$.
 \end{proposition}

Propositions \ref{p.HL235} and \ref{p.HLconverse} are converse to each other
and this correspondence is canonical, so we immediately obtain a natural isomorphism:
 \begin{align} \label{E:CDsym}
 \aut(\sC)_x \cong \aut(D)_y, \quad \forall y \in Y, \quad \forall x \in C_y.
 \end{align}

The vertical distributions $E = \ker(d\mu)$ and $V = \ker(d\rho)$ define line distributions on $\cC$ and we have the following double fibration:

\begin{center}
\begin{figure}[h]
\begin{tikzcd}
& (\cC = \bbP D;E,V) \arrow[ld,"\rho"'] \arrow[rd,"\mu"] \\ (Y,D) & & (X,H) \end{tikzcd}
\label{F:fibration}
\caption{Canonical double fibration relating $(2,3,5)$-distributions and Legendrian cone structures of VMRT type}
\end{figure}
\end{center}

On the 6-manifold $\cC = \bbP D$, there is also a canonical rank 2 distribution $\wD$ that is tautologically induced: given $\ell \in \bbP D_x$, define
\begin{align}
\wD_\ell := (d\rho)^{-1}( \,\widehat{\ell} \,) \subset T_\ell (\bbP D).
\end{align}
 Let $\wD^k$ denote the $k$-th weak derived system of $\wD$.  Its {\sl Cauchy characteristic space} is $\Ch(\wD^k) := \{ \bX \in \Gamma(\wD^k) :  \cL_\bX \wD^k \subset \wD^k \}$.


\begin{proposition}\label{p.bracket}
Given notation as above, we have:
\begin{enumerate}
\item $\wD = E \oplus V$ and $\wD^2 = (d\rho)^{-1} D$.
\item $V = \Ch(\wD^2) \cap \wD$.
\item $E = \Ch(\wD^4) \cap \wD$.
\end{enumerate}
In particular, the successive brackets of sections of the distribution $\wD = E \oplus V$ generate the tangent bundle $T\sC$. \end{proposition}

 \begin{proof}
 Let $\{ e_1,e_2,e_3,e_4,e_5 \}$ be a local framing adapted to the weak derived flag of $D \subset TY$, i.e.\ $D = \langle e_1, e_2 \rangle$, $D^2/D \equiv \langle e_3 \rangle$, $D^3 / D^2 \equiv \langle e_4, e_5 \rangle$ with
 \begin{align}
 [e_1,e_2] = e_3, \quad [e_1, e_3] = e_4, \quad [e_2, e_3] = e_5.
 \end{align}
 In a local trivialization of $\bbP D \to Y$ about a given $y \in Y$, we introduce an affine fibre coordinate $t$ so that $\ell \in \bbP D_y$ corresponds to $\langle e_1 + t e_2 \rangle$.  Then $V = \langle \partial_t \rangle$, $\wD = \langle \partial_t, e_1 + t e_2 \rangle$, and $\wD^2 = \langle \partial_t, e_1, e_2 \rangle = (d\rho)^{-1} D$.  Continuing,
 \begin{align}
 \wD^3 / \wD^2 \equiv \langle e_3 \rangle, \quad \wD^4 \equiv \langle e_4 + t e_5 \rangle.
 \end{align}
 The Lie bracket induces a tensorial map $\wD \otimes \wD^4 / \wD^3 \to \wD^5 / \wD^4$ that is surjective and has rank 1 kernel, denoted by  $\widetilde{E} \subset \wD$, distinct from $V$.

 In Proposition \ref{p.HLconverse} (or \cite[Theorem 5.10]{HLGC}), the contact manifold $(X, H)$ is constructed precisely by taking the quotient by the rank 1 kernel $\widetilde{E} \subset \wD$.  Thus, $\widetilde{E} = \ker(d\mu) = E$ follows from the definition of $\mu$.
 \end{proof}

The triple $(\bbP D; E,V)$ is called a {\sl pseudo-product structure of $G_2$-type}.  We observe that the symbol algebra of $\wD$, i.e.\ the associated-graded of the weak derived flag of $\wD$, has depth 5 and is isomorphic to $\fg_-$ for $\fg$ of type $G_2$ in the grading associated to the Borel subgroup $P_{1,2}$ (see Figure \ref{F:G2pr}).  Integral curves of $E$ are referred to as {\sl abnormal extremals}, and the quotient $X = \bbP D / E$ is the {\sl abnormal extremal space}, which is equipped with the contact distribution $H := \wD^4 / E \subset TX$.  Along the fibres of the projection $\mu : \bbP D \to X$, the line field $V$ twists and generates a Legendrian cone structure $\cC \subset \bbP H$.

\begin{corollary}\label{c.bracket}
In Proposition \ref{p.bracket}, suppose there exists a dense open subset $\sC^o \subset \sC$ such that for any two points $y_1, y_2 \in \rho(\sC^o \cap \mu^{-1}(x))$ for $ x \in X$
(resp. $x_1, x_2 \in \mu(\sC^o \cap \rho^{-1}(y))$ for $ y \in Y$), the germ of $\sD$ at $y_1$ and the germ of $\sD$ at $y_2$ are equivalent (resp. the germ of $\sC$ at $x_1$ and the germ of $\sC$ at $x_2$ are equivalent). Then there exists a dense open subset $Y^o \subset Y$ (resp. $X^o \subset X$) such that the germ of $D$ at any two points $y_1, y_2 \in Y^o$ (resp. the germ of $\sC$ at any two points $x_1, x_2 \in X^o$) are equivalent. \end{corollary}

\begin{proof} Consider the equivalence relation  on points of $Y$ given by $y_1 \sim y_2$ if and only if the germ of $D$ at $y_1$ and the germ of $D$ at $y_2$ are equivalent. By the assumption two points in $Y$ joined by the $\rho$-image of a fiber of $\sC^o \to X$ are equivalent. But Proposition \ref{p.bracket}, two general points of $Y$ can be joined by a connected chain of $\rho$-images of fibers of $\sC^o \to X$. Thus we can find $Y^o \subset Y$ with the desired property. The argument for $X^o \subset X$ is similar. \end{proof}

The following is a special case of Main Theorem in Section 2 of \cite{Mok}.

\begin{theorem}\label{t.Mok}
Let $\sC \subset \BP H$ be a Legendrian cone structure of VMRT-type on a contact manifold $(X, H)$ of dimension $5$. If the fiber $\sC_x \subset \BP H_x$ is isomorphic to (the germ of) the rational normal curve $\mathbf{Z} \subset \BP \Sym^3 W$ for some $x \in X$, then there exists an open neighborhood $x \in U \subset X$ such that $\sC|_U$ is $\mathbf{Z}$-isotrivial and  locally flat. \end{theorem}

 We have the following construction from Theorem 1.3 of \cite{HMSJ}.

 \begin{theorem} \label{t.Heisenberg}
 Let $Z \subset \BP^3$ be a germ of a  Legendrian curve. Then there exists a contact manifold $(X, H)$ of dimension $5$ with a Legendrian cone structure $\sC \subset \BP H$ of VMRT-type with the associated $\BP^1$-fibration $\rho: \sC \to Y$ such that:
 \begin{itemize}
 \item[(i)] the 5-dimensional Heisenberg group $\Heis$ acts on $X$ with an open orbit $X^o \subset X$ giving an identification $\Heis \cong X^0$;
 \item[(ii)] the restriction $\sC|_{X^o}$ is biholomorphic to $\sC^Z \subset \BP \sH$ in Definition \ref{d.flatcone};
 \item[(iii)] the  lines on $X$ corresponding to fibers of $\rho$ intersecting $\Heis$ on affine  lines of $\Heis$ with respect to the affine coordinates on $\Heis$ in Proposition \ref{p.Heis};
 \item[(iv)] by (iii),  we can identity $Y$ with the set of affine lines on $\Heis$ that are in the direction of $Z \subset \BP \sH_o$ up to left-translation.
 \end{itemize}
 \end{theorem}

\begin{proposition}\label{p.dim}

When the germ $Z$ in Theorem \ref{t.Heisenberg} is nondegenerate and nonhomogeneous, the associated $(2,3,5)$-distribution $D$ corresponding to the locally flat $\sC \subset \BP H$ of VMRT type satisfies $\dim \aut(D)^0_y = 2$. \end{proposition}

\begin{proof}
From Proposition \ref{p.Heis} (iv) and Theorem \ref{t.Heisenberg} (ii), the Lie algebra  $\aut(\sC)_x, x \in X^o,$ of infinitesimal automorphisms of the cone structure  $\sC \subset \BP H$ contains the Lie algebras of   $\Heis$ and the multiplicative group $\C^\times$.

 When $Z$ is nonhomogeneous, we have $\dim \aut(\sC)_x^0 =1$ for $x \in X^o$ and consequently $\dim \aut(\sC)_x =6$. Since $\aut(\sC)_x^0$ acts trivially on $\sC_x$, the induced action of $\aut(\sC)_x$ on the 6-dimensional manifold $\sC$ has  5-dimensional orbits.  The fibers of $\rho: \sC \to Y$ are contained in these 5-dimensional $\aut(\sC)_x$-orbits because  each affine line in $\Heis$ tangent to $\sH$ is an orbit of a 1-parameter subgroup of $\Heis$.

Thus, $D \subset TY$ is nontransitive and the orbit-dimension of $\aut(D)$ is 4. It follows that
 \begin{align}
 \dim \aut(D)^0 =  \dim \aut(D) -  4 = \dim \aut(\sC) -4 =  2.
 \end{align}
\end{proof}

\section{$(2,3,5)$-distributions: symmetries and Legendrian curves}
\label{S:SymLeg}

 A well-known symmetry gap result \cite{Car1910, KT2014} for $(2,3,5)$-distributions is:

 \begin{theorem}\label{t.G2}
Let $D \subset TM$ be a $(2,3,5)$-distribution.
If $\dim \aut(D)_y \geq 8$ for some $y \in M$, then $\aut(D)_y$ is isomorphic to the simple Lie algebra of type $G_2$ of dimension 14 and $D$ is flat.
 \end{theorem}

 Here, flatness refers to vanishing of the curvature of the corresponding regular, normal Cartan geometry of type $(G_2,P_1)$.  Otherwise:

 \begin{theorem}\label{t.KT}
Let $D \subset TM$ be a $(2,3,5)$-distribution that is not flat.  For a general $y \in M$, the action of a nonzero element $\vec{v} \in \aut(D)_y^0$ on $\BP D_y$ is nontrivial.
 \end{theorem}

 This result is implicitly contained  in  \cite[Thm.4.2]{KT2018},  the fact that the symmetry algebra is 1-jet determined. For the reader's convenience, we provide a Cartan-geometric proof in Appendix \ref{A:235nf}.  These results will be used in Section \ref{S:MainResults} in the proofs of our main results.\\

 We turn now to (complex) multiply-transitive $(2,3,5)$-distributions.  These were classified in \cite{Car1910, DG} (see \cite{The2022} for a completeness argument).  Aside from the flat model, there are three classes: $\sfN7_c, \sfN6, \sfD6_a$ (with $c^2, a^2 \in \bbC$ the invariants classifying the structure), and Table \ref{F:Fstr} below is a Lie-theoretic presentation of the symmetry algebra $\ff$ as given in \cite[Table 6]{The2022}.  For each, $\ff$ is equipped with a decreasing filtration $\ff = \ff^{-3} \supset \ff^{-2} \supset \ff^{-1} \supset \ff^0$ and an adapted basis with
 \begin{align}
 \ff^{-3} / \ff^{-2} = \langle X_4, X_5 \rangle, \quad \ff^{-2} / \ff^{-1} = \langle X_3 \rangle, \quad \ff^{-1} / \ff^0 = \langle X_1, X_2 \rangle,
 \end{align}
 and the isotropy subalgebra $\ff^0 \subset \ff$ consisting of the remaining basis elements.  We get an $\ff^0$-invariant $(2,3,5)$-filtration on $\ff / \ff^0$.

 \begin{center}
 \begin{tiny}
 \begin{table}[h]
 \[
 \begin{array}{|c|l|}\hline
 \mbox{Label} & \mbox{Lie bracket on $\ff$, calculated via $[\cdot,\cdot]_\ff = [\cdot,\cdot] - \kappa(\cdot, \cdot)$}\\ \hline\hline
 \sfN7_c &
 \begin{array}{c|ccccccc}
 [\cdot,\cdot]_\ff & T & N  & X_1 & X_2 & X_3 & X_4 & X_5\\ \hline
 T & \cdot & -N & \cdot & -X_2 & -X_3 & -X_4 & -2 X_5\\
 N && \cdot & X_2 & \cdot & \cdot & -X_5 & \cdot\\
 X_1 &&&  \cdot & -3c N - 2 X_3 & -2c X_2 + 3 X_4 & -N + cX_3 & \cdot\\
 X_2 &&& & \cdot & -3 X_5 & \cdot & \cdot \\
 X_3 &&&&&\cdot & \cdot & \cdot\\
 X_4 &&&&&&\cdot & \cdot\\
 X_5 &&&&&&&\cdot
 \end{array}\\ \hline
 \sfN6 & \begin{array}{c|ccccccc}
 [\cdot,\cdot]_\ff & N  & X_1 & X_2 & X_3 & X_4 & X_5\\ \hline
 N & \cdot & X_2 & -2N & \cdot & -X_5 + N & \cdot\\
 X_1 &&  \cdot & -18N + 2X_1- 2X_3 & -12X_2 + 3 X_4 & -2X_1 + 6X_3 - 42 N & -X_4\\
 X_2 && & \cdot & 27 N - 3 X_5 & - X_2 - X_4 & -N + X_5 \\
 X_3 &&&&\cdot & -60 N + 6 X_3 & \cdot\\
 X_4 &&&&&\cdot & -24 N + 2 X_3 + 4 X_5\\
 X_5 &&&&&&\cdot
 \end{array}\\ \hline
 \sfD6_a & \begin{array}{c|cccccc}
 [\cdot,\cdot]_\ff & T & X_1 & X_2 & X_3 & X_4 & X_5\\ \hline
 T & \cdot & X_1 & -X_2 & \cdot & X_4 & -X_5\\
 X_1 &&  \cdot & 3aT - 2X_3 & \,\,\,\,2aX_1 + 3X_4  & \cdot & 6T - aX_3 \\
 X_2 && & \cdot & - 2aX_2 - 3X_5 & - 6 T + aX_3 & \cdot\\
 X_3 &&&&\cdot & -(a^2+3) X_1 & (a^2+3) X_2\\
 X_4 &&&&&\cdot & a(a^2-1) T - 2 X_3 \\
 X_5 &&&&&&\cdot
 \end{array}\\ \hline
 \end{array}
 \]
 \caption{Multiply-transitive $(2,3,5)$-structures: $\ff^{-1} / \ff^0 = \langle X_1, X_2 \rangle$, $\ff^{-2} / \ff^{-1} = \langle X_3 \rangle$, $\ff^{-3} / \ff^{-2} = \langle X_4, X_5 \rangle$.}
 \label{F:Fstr}
 \end{table}
 \end{tiny}
 \end{center}

 Let us lift these homogeneous structures $(M;D)$ to $(\bbP D;E,V)$.  We can do so algebraically via a generic choice of $\ell \in \bbP D|_o \cong \bbP(\ff^{-1} / \ff^0)$.  This induces on $\widetilde\ff = \ff$ a decreasing filtration $\widetilde\ff = \widetilde\ff^{-5} \supset ... \supset \widetilde\ff^{-1} \supset \widetilde\ff^0$ so that $\wD = E\oplus V$ satisfies  $\wD^i|_\ell \cong \widetilde\ff^{-i} / \widetilde\ff^0$ for $i \geq 1$.  See Table \ref{F:lifted-str} for the associated data.

 \begin{table}[h]
 \[
 \begin{array}{|c|c|c|c|c|c|c|c|}\hline
 \mbox{Model} & \ell \in \bbP(\ff^{-1} / \ff^{-0}) & \widetilde\ff^0 & E & V & \widetilde\ff^{-2} / \widetilde\ff^{-1} & \widetilde\ff^{-3} / \widetilde\ff^{-2} & \widetilde\ff^{-4} / \widetilde\ff^{-3} \\ \hline\hline
 \sfN7_c & X_1 & T & X_1 & N & X_2 & X_3 & X_4 \\
 \sfN6 & X_1 & \cdot & X_1 & N & X_2 & X_3 & X_4 \\
 \sfD6_a & X_1 + X_2 & \cdot & X_1 + X_2 & T & X_1 & X_3 & X_4 - X_5 \\ \hline
 \end{array}
 \]
 \caption{Multiply-transitive $(2,3,5)$-data lifted to $\bbP D$}
 \label{F:lifted-str}
 \end{table}

 On the leaf space $X = \bbP D / E$, $H = \wD^4 / E$ is a contact distribution equipped with a nondegenerate Legendrian cone structure $\cC \subset \bbP H$.  Its fibers $\sC_x$ are homogeneous: a symmetry generator $A$ in \eqref{E:wZ} arises from the infinitesimal action of $E$ on $H \cong \widetilde\ff^{-4} / (\widetilde\ff^0 + E)$ through the $A$-admissible base point $V \!\!\mod (\widetilde\ff^0 + E)$.  We compute the associated matrix for $A$ in the specified basis, and its minimal / characteristic polynomial $f_A(s)$.  Using \eqref{E:CI}, $q_0$ and $\cI$ can be efficiently computed from the coefficients of $f_A(s)$.  This yields Table \ref{F:MTinv}.

 \begin{center}
 \begin{footnotesize}
 \begin{table}[h]
 \[
 \begin{array}{|c|c|c|c|c|c|} \hline
 \mbox{Model} & \mbox{Basis} & A & f_A(s) & q_0 & \cI\\ \hline\hline
 \sfN7_c & N, X_2, X_3, X_4 &
 \begin{psm}
 0 & -3c & 0 & -1\\
 -1 & 0 & -2c & 0\\
 0 & -2 & 0 & c\\
 0 & 0 & 3 & 0
 \end{psm} & s^4 - 10 c s^2 + 9 c^2 + 6 & \neq 0 & -\frac{c^2}{6} \\ \hline
 \sfN6 & N, X_2, X_3, X_4 &
 \begin{psm}
 0 & -18 & 0 & -42\\
 -1 & 0 & -12 & 0\\
 0 & -2 & 0 & 6\\
 0 & 0 & 3 & 0
 \end{psm} & s^4 - 60 s^2 + 576 & \neq 0 &
 -\frac{1}{7}\\ \hline
 \sfD6_a & T, X_1, X_3, X_4 - X_5 &
 \begin{psm}
 0 & -3 a & 0 & -12\\
 -2 & 0 & 4a & 0\\
 0 & 2 & 0 & 2a\\
 0 & 0 & 3 & 0
 \end{psm} & s^4 - 20 a s^2 + 36 a^2 - 144 & \neq 0 &
 \frac{a^2}{36}\\ \hline
 \end{array}
 \]
 \caption{Symmetry generator $A$ for associated Legendrian curve in $\bbP^3$ arising from the $E$-action on $H$}
 \label{F:MTinv}
 \end{table}
 \end{footnotesize}
 \end{center}

 \begin{proposition} For all multiply-transitive $(2,3,5)$-distributions, the corresponding nondegenerate Legendrian cone structure $(X,H,\cC)$ of VMRT type is $Z$-isotrivial and multiply-transitive.  For $\sfN7_c$, it is locally flat, while for $\sfN6$ and $\sfD6_a$ it is not.  If $c^2 = -\frac{a^2}{6}$, then the Legendrian curves $Z$ associated to $\sfN7_c$ and $\sfD6_a$ are homogeneous and projectively equivalent.  If moreover $c^2 = \frac{6}{7}$, then this is also projectively equivalent to the Legendrian curve associated to $\sfN6$.
 \end{proposition}

 \begin{proof} The $(q_0,\cI)$ data computed in Table \ref{F:MTinv} classifies via Theorem \ref{T:homZ} the associated homogeneous non-degenerate Legendrian curves $Z \subset \bbP^3$.  Assertions of local (non-)flatness follows from \eqref{E:CDsym} and Proposition \ref{p.flat}(ii).
 \end{proof}

 \begin{remark}\label{r.Heis}
 It is known \cite{DK2014} by an explicit Lie algebra computation that the symmetry algebra of  the $\sfN7_c$ case is isomorphic to $\heis \rtimes \bbC^2$.  The local flatness of the corresponding cone structure implies that it is locally the $Z$-isotrivial structure $(\Heis, \sH, \sC^Z)$.  This gives a conceptual explanation why the symmetry algebra is isomorphic to $\heis \rtimes \bbC^2$.  Namely, $\heis$ arises from infinitesimal Heisenberg translations on $\Heis$, while the abelian subalgebra $\bbC^2$ arises from symmetries of the homogeneous Legendrian curve $Z \subset \bbP^3$ (which is not rational normal).  The action of $\bbC^2$ on $\heis$ depends on the parameter $c$.
 \end{remark}

 As an example, consider two (real) 2-spheres with ratio of radii $\rho > 0$ rolling on each other without twisting or slipping \cite{Agr2007}, \cite{BH2014}, \cite{BM2009}.  There is an associated multiply-transitive $(2,3,5)$-distribution called the {\sl rolling distribution}.  From \cite[(5.9)]{The2022}, it is generally a real form of $\sfD6_a$
with the relation
 \begin{align}
 \frac{a^2}{36} = \frac{(\rho^2+1)^2}{(\rho^2 - 9)(9\rho^2 - 1)}, \quad \rho^2 \not\in \{ 9, \tfrac{1}{9} \}.
 \end{align}
 The associated homogeneous nondegenerate Legendrian cone structure $\cC \subset \bbP H$ is $Z$-isotrivial, with (the complexification of) $Z \subset \bbP^3$ having $\cI = \frac{a^2}{36}$ for $\sfD6_a$.  From the formula $\cI = \tfrac{(r^2+1)^2}{(r^2-9)(9r^2-1)}$ in Theorem \ref{T:homZ}, we have $r^2 = \rho^2$ or $\rho^{-2}$.  Thus, it is of type $\bL_{\rho^2}$, i.e.\ it is (complex) projectively equivalent to $Z \subset \bbP^3$ arising from $\gamma(t) = \exp(tA) z$, where
 \begin{align}
 A = \diag(\rho,1,-1,-\rho), \quad z = (1,1,1,1)^\top, \quad \rho^2 \not\in \{ 9, \tfrac{1}{9} \}.
 \end{align}
 The exceptional cases $\rho = 3$ or $\rho = \tfrac{1}{3}$ occur for the rational normal curve $\bZ$.

\section{Proofs of Main Results}
\label{S:MainResults}

\begin{proposition}\label{p.trans}
Let $D \subset TM$ be a non-flat $(2,3,5)$ distribution and $(X,H,\sC)$ the associated nondegenerate Legendrian cone structure of VMRT type.
 \begin{itemize}
 \item[(i)] If $\dim \aut(D)^0_y \geq 1 $ for a general point $y \in M$, then $\dim \aut(\sC)_x^0 \geq \dim \aut(D)_y^0 -1 $ for a general point $ x\in X$ and $\sC \subset \BP H$ is transitive.
 \item[(ii)] If $\aut(\sC)^0_x$ acts nontrivially on $\sC_x$ for a general $x \in X$,  then $D \subset T Y$ is transitive.
 \item[(iii)] If $\dim \aut(\sC)^0_x \geq 2$, then $D \subset TY$ is transitive.
 \end{itemize}
 \end{proposition}

\begin{proof}
By Theorem \ref{t.KT}, the Lie algebra  $\aut(D)_y^0$ for a general $y \in Y$ acts effectively on $\BP D_y \cong C_y \subset X$. Thus $\dim \aut(\sC)_x^0 \geq \dim \aut(D)_y^0 -1 $ for a general point $ x\in X$.
Moreover, by Corollary \ref{c.bracket}, there exists a dense open subset $X^o \subset X$ such that the germ of $\sC$ at any two points of $X^o$ are equivalent. This implies that $\sC \subset \BP H$ is transitive, proving (i).

 If $\aut(\sC)^0_x$ acts nontrivially on $\sC_x$ for a general $x \in X$,
then by Corollary \ref{c.bracket}, there exists a dense open subset $Y^o \subset Y$ such that the germ of $D \subset TY$ at any two points of $Y^o$ are equivalent. This implies that $D \subset TY$ is transitive, proving (ii).

If $\dim \aut(\sC)_x^0 \geq 2$, then some element of $\aut(\sC)_x^0$ acts nontrivially on $\sC_x$. Thus (iii) follows from (ii). \end{proof}

\begin{proof}[Proof of Theorem \ref{t.aut}] Let $y \in M$ be a general point.
If $\dim \aut(D)^0_y \geq 1$, then by the transitivity of $\sC \subset \BP H$ in Proposition \ref{p.trans}, we have $\dim \aut(D) = \dim \aut(\sC) \geq 5$.
This proves (i).
If $\dim \aut(D)_y^0 \geq 2$, Proposition \ref{p.trans} gives  $\dim \aut(\sC)_x^0 \geq 1$ for a general point $ x\in X$.  Thus $\dim \aut(\sC) \geq 6$, which implies $\dim \aut(D) \geq 6$, proving (ii).
If $\dim \aut(D)_y^0 \geq 3$, Proposition \ref{p.trans} gives $\dim \aut(\sC)_x^0 \geq 2$. Thus both $\sC \subset \BP H$ and  $D \subset TY$ are transitive by Proposition \ref{p.trans}. Hence $$\dim \aut(D) \geq 5 + \dim \aut(D)^0_y \geq 8.$$ This implies that $D$ is flat, proving (iii).

If $D$ is nontransitive, then $5 = \dim M >  \dim \aut(D) - \dim \aut(D)^0.$  If furthermore $\dim \aut(D) =6$, then $\dim \aut(D)^0 > \dim \aut(D) - 5 =1$ and (iii) give $\dim \aut(D)^0 =2.$   This proves (iv).
\end{proof}

\begin{proof}[Proof of Theorem \ref{t.1to1}]
Since $D$ is nontransitive and $\dim \aut(D) =6$, we have $\dim \aut(D)^0 =2$ by Theorem \ref{t.aut}. Thus $\sC$ is transitive and $\aut(\sC)_x^0 \neq 0$ for a general $x \in X$ by Proposition \ref{p.trans}. By Proposition \ref{p.trans} (ii), we know that $\aut(\sC)_x^0$ acts trivially on $\sC_x$. Thus $\sC \subset \BP H$ is locally flat by Proposition \ref{p.HL} and $6= \dim \aut(\sC) = 5 + \dim \finf{Z}$.  Thus  $Z$ is not homogenous.

Given an $\aut$-generic point $y \in Y$, pick a general point $x \in C_y \subset X$ and consider the germ of the nonhomogeneous Legendrian curve $\sC_x \subset \BP H_x$ at the point $[\BP T_x C_y]$. The equivalence class (up to contacto-isomorphisms of $\BP H_x$) of this germ of Legendrian curves does not depend on the choice of $x$ by the local flatness of $\sC \subset \BP H$. Conversely, given a nonhomogeneous Legendrian curve $Z \subset \BP^3$ and a nondegenerate  point $z \in Z$, we use Theorem \ref{t.Heisenberg} to find a germ of  $(2,3,5)$-distributions $D \subset TY$  with $\dim \aut(D) =6$.  Then $\dim \aut(D)^0_y =2$ from Proposition \ref{p.dim}. It is clear that this gives a one-to-one correspondence between the equivalence classes.
\end{proof}

In Theorem \ref{T:homZ}, we established that for a nondegenerate Legendrian curve $Z \subset \bbP^3$, $\dim \finf{Z} = 3$ is impossible, but this required a straightforward (but tedious) check that $\dim \finf{Z} = 2$ in the homogeneous case when $Z \not\cong \bZ$.  Below, we confirm this as an easy consequence of our constructions and the symmetry gap for $(2,3,5)$-distributions.

\begin{proposition} \label{P:LCgap}
Let $Z \subset \bbP^3$ be a germ of a nondegenerate Legendrian curve with $\dim \finf{Z} \geq 3$.  Then $Z$ is isomorphic to the germ of a rational normal curve $\bZ$ and $\dim \finf{Z} = 4$.
\end{proposition}

 \begin{proof}
By Theorem \ref{t.Heisenberg}, we have a $Z$-isotrivial  Legendrian cone structure $\sC \subset \BP H$ of VMRT-type on a contact manifold $(X, H)$ that is locally flat on a dense open subset $U \subset X$.  Let $D \subset TY$ be the associated $(2,3,5)$-distribution.  Assume that $Z \not\cong \bZ$.  By \eqref{E:CDsym} and Proposition \ref{p.flat}, $\dim \aut(D)_y = \dim \aut(\sC)_x = 5 + \dim \finf{Z}$ for $y \in Y$ with $x \in C_y \subset U$.  If $\dim \finf{Z} \geq 3$, then $\dim \aut(D)_y \geq 8$.  By Theorem \ref{t.G2}, $D \subset TY$ is flat and so $Z \cong \bZ$.
 \end{proof}

\section*{Acknowledgments}

We thank Boris Doubrov for enlightening discussions, particularly during a visit to Daejeon in March 2024.  Jun-Muk Hwang's research has been supported by the Institute for Basic Science (IBS-R032-D1).  Dennis The's research leading to these results has received funding from the Norwegian Financial Mechanism 2014-2021 (project registration number 2019/34/H/ST1/00636), the Troms\o{} Research Foundation (project ``Pure Mathematics in Norway''), the UiT Aurora project MASCOT, and this article/publication is based upon work from COST Action CaLISTA CA21109 supported by COST (European Cooperation in Science and Technology), \href{https://www.cost.eu}{https://www.cost.eu}.

\appendix

\section{Some Maple code}
 \subsection{The relative invariant $\cR$ has weight 10}
 \label{S:R}

We saw in Proposition \ref{P:relinv}(ii) that $\cR = 8 q_0 q_0'' - 9 (q_0')^2$ has weight 10.  Here is Maple code that establishes this:

\begin{verbatim}
restart: with(DifferentialGeometry): with(JetCalculus):
Preferences("JetNotation", "JetNotation2"):
DGsetup([t],[q0],J,2):
DGsetup([T],[Q0],K,2):
L:=(a*t+b)/(c*t+d):
DL:=diff(L,t):
phi:=Transformation(J,K,[T=L,Q0[0]=q0[0]/DL^4]):
phi2:=Prolong(phi,2):
newR:=8*Q0[0]*Q0[2]-9*Q0[1]^2:
simplify(Pullback(phi2,newR));
\end{verbatim}
This yields the result
\begin{align}
\frac{(ct + d)^{20}}{(ad-bc)^{10}} (8 q_0 q_0'' - 9 (q_0')^2),
\end{align}
which is the same as $\frac{\cR}{(\lambda')^{10}}$.

 \subsection{Transformation to Laguerre--Forsyth canonical form}
 \label{S:LF}

 We describe here how \eqref{E:LFtransf} was found, defining a transformation $(\widetilde{t},\widetilde{u}) = (\lambda(t), \mu(t) u)$ that brings \eqref{E:homODE} to Laguerre--Forsyth canonical form.  This is easily accomplished with the aid of Maple.
 \begin{verbatim}
 restart: with(DifferentialGeometry): with(JetCalculus):
 Preferences("JetNotation", "JetNotation2"):
 DGsetup([t],[u],J,4): DGsetup([T],[U],K,4):
 phi:=Transformation(J,K,[T=lambda(t),U[0]=mu(t)*u[0]]):
 phi4:=Prolong(phi,4):
 newODE:=U[4]+Q0(T)*U[0]:
 oldODE:=simplify(Pullback(phi4,newODE)):
 cf:=i->simplify(diff(oldODE,u[i])):
 0=expand(cf(3)/cf(4));
 \end{verbatim} \vspace{-0.1in}
 The last line comes from \eqref{E:homODE} having no $u'''$ term.  We obtain
 \begin{align}
 0 = -\frac{6\lambda''}{\lambda'} + \frac{4\mu'}{\mu},
 \end{align}
 which has solution $\mu = r (\lambda')^{3/2}$, so we continue with:
 \begin{verbatim}
 eval(cf(2)/cf(4),mu(t)=r*diff(lambda(t),t)^(3/2)):
 -a^2-b^2=expand(%);
 \end{verbatim} \vspace{-0.1in}
 This yields:
 \begin{align}
 -a^2-b^2 = \frac{5 \lambda'''}{\lambda'} - \frac{15 (\lambda'')^2}{2 (\lambda')^2}
 \end{align}
 Solving this ODE (using the \verb|dsolve| command) yields a 3-parameter family of solutions.   When $a^2+b^2 \neq 0$, \eqref{E:LFtransf} is one such solution.

\section{Proof of Theorem \ref{t.KT}}
\label{A:235nf}

\subsection{Prolongation-rigidity} Let $G$ be a semisimple Lie group and $P \subset G$ a parabolic subgroup.  Let $\fg = \fg_{-\mu} \oplus ... \oplus \fg_\mu$ be the associated Lie algebra grading induced by a grading element $\sfZ \in \fz(\fg_0)$, with parabolic subalgebra $\fp = \fg_{\geq 0}$ the Lie algebra of $P$, and associated $P$-invariant filtration $\fg = \fg^{-\mu} \supset ... \supset \fg^0 \supset ... \supset \fg^\mu$.

 \begin{definition}
 Given a $\fg_0$-representation $\bbV$, and any $\phi \in \bbV$, we  define the {\sl extrinsic Tanaka prolongation} $\fa^\phi \subset \fg$ as the graded subalgebra with:
 \begin{enumerate}
 \item[(i)] $\fa^\phi_{\leq 0} := \fg_- \oplus \fann(\phi)$.
 \item[(ii)] $\fa^\phi_i := \{ x \in \fg_i : [x,\fg_{-1}] \subset \fa^\phi_{i-1} \}$ for all $i > 0$.
 \end{enumerate}
 \end{definition}

Consider the chain complex $C_\bullet = \bigwedge^\bullet \fg_+ \otimes \fg$ with differential $\partial^* : C_\bullet \to C_{\bullet-1}$  (from page 262 of \cite{CS2009}).  The quotient of 2-cycles (i.e.\ {\sl normal} elements) modulo 2-boundaries is the homology space $H_2(\fg_+,\fg)$. 
 Let $H_2(\fg_+,\fg)^1 \subset H_2(\fg_+,\fg)$ be the subspace on which $\sfZ$ acts with positive eigenvalues.

 \begin{definition}
 We say that $(G,P)$ is {\sl prolongation-rigid} if $\fa^\phi_+ = 0$ for any nonzero element $\phi$ of the $\fg_0$-representation  $H_2(\fg_+,\fg)^1$.
 \end{definition}

 Prolongation-rigidity was introduced in \cite[\S 3.4]{KT2014} and investigated there in detail.  Relevant to us for $(2,3,5)$-distributions is the following special case of \cite[Corollary 3.4.8]{KT2014}.

 \begin{lemma}\label{l.PR}
 $(G_2,P_1)$ is prolongation-rigid.
 \end{lemma}

\subsection{A constraint on symmetries}

 Given a $(2,3,5)$-distribution, there is a uniquely associated regular, normal parabolic geometry $(\cG \to M, \omega)$ of type $(G,P) = (G_2,P_1)$, where $\cG$ is a $P$-principal bundle and $\omega$ is a $\fg$-valued Cartan connection.  Its curvature function is a $P$-equivariant function $\kappa : \cG \to \bigwedge^2(\fg/\fp)^* \otimes \fg \cong \bigwedge^2 \fg_+ \otimes \fg$, where the $P$-equivariant isomorphism is induced from the Killing form on $\fg$.  By regularity and normality, we have $\kappa \in \ker(\partial^*)^1$.  Its quotient by $\operatorname{im}(\partial^*)^1$ yields the harmonic curvature $\kappa_H : \cG \to H_2(\fg_+,\fg)^1$, which completely obstructs flatness of the geometry, i.e.\ $\kappa_H \equiv 0$ if and only if $\kappa \equiv 0$.

  Given any $u \in \cG$, we let $\ff(u) \subset \fg$ be the image of $\omega_u$ restricted to the infinitesimal symmetry algebra $\inf(\cG,\omega) = \{ \xi \in \fX(\cG)^P : \cL_\xi \omega = 0 \}$, where $\fX(\cG)^P$ denote $P$-invariant vector fields on $\cG$, and $\ff(u)$ inherits a depth 3 filtration
 \[
  \ff^{-3}(u) \supset \cdots \supset \ff^0(u) \supset \cdots \supset \ff^{3}(u)
 \]
 from that on $\fg$.  Although $\ff(u)$ is a filtered Lie algebra, it is generally not a Lie subalgebra of $\fg$.  Indeed, we have
 \begin{align}
 [\cdot,\cdot]_{\ff(u)} = [\cdot,\cdot]_\fg - \kappa_u(\cdot,\cdot).
 \end{align}
  However, when passing to the associated-graded $\fs(u) = \tgr(\ff(u))$, regularity shows that this is a graded subalgebra of $\fg$.  A key constraint is
 \begin{align} \label{E:SA}
 \fs(u) \subseteq \fa^{\kappa_H(u)}, \quad \forall u \in \cG.
 \end{align}
 (This was proven on the open dense subset of so-called ``regular points'' of $\cG$ in \cite[Thm.2.4.6]{KT2014}, and was improved to all points of $\cG$ in \cite[Thm.3.3]{KT2018}.)

 \subsection{Final steps}

 Let $(M,D)$ be a $(2,3,5)$-distribution and $(\cG \stackrel{\pi}{\to} M, \omega)$ its associated regular, normal Cartan geometry of type $(G,P) = (G_2,P_1)$.  Assuming non-flatness, $\kappa_H$ is non-vanishing somewhere, hence by continuity on an open set.   Given a general point $y \in M$, we may assume that
 \begin{align} \label{E:KH}
 \kappa_H(u) \neq 0, \quad \forall u \in \pi^{-1}(y).
 \end{align}

 Let $\vec{v}$ be a nonzero symmetry of $(M,D)$ vanishing at $y$, with corresponding $\xi \in \inf(\cG,\omega)$.  Fix any $u \in \pi^{-1}(x)$.  The natural $P$-filtration on $\fg$ induces a filtration $T^{-3} \cG \supset ... \supset T^3 \cG$ of $T\cG$.  (In particular, $T^{-1}_u \cG$ surjects onto the fibre $D_y = T^{-1}_y M$ via the differential $\pi_*$.)  Let $\eta \in \Gamma(T^{-1}\cG)^P$ be arbitrary.  Let $0 \neq X = \omega(\xi(u)) \in \fg^0 = \fp$ and $Y = \omega(\eta(u)) \in \fg^{-1}$.  Then the symmetry condition implies
 \begin{align}
 0 &= (\cL_\xi \omega)(\eta) = (\iota_\xi d\omega)(\eta) + d(\omega(\xi))(\eta) \\
 &= (d\omega)(\xi,\eta) + \eta \cdot (\omega(\xi)) = \xi \cdot \omega(\eta) - \omega([\xi,\eta]),
 \end{align}
 and so if $\zeta_X$ is the vertical vector field generated by $X$, then we have
 \begin{align}
 \omega([\xi,\eta])(u) &= (\xi \cdot \omega(\eta))(u) = (\zeta_X \cdot \omega(\eta))(u) \\
 &= \left.\frac{d}{dt} \right|_{t=0} \omega(\eta)(u \cdot \exp(tX)) \\
 &=  \left.\frac{d}{dt} \right|_{t=0} \operatorname{Ad}_{\exp(-tX)} (\omega(\eta)(u)) = -[X,Y]
 \end{align}
 To complete the proof, we need to show that
 \begin{align} \label{E:X}
 Y \mod \fp \quad\mapsto\quad [X,Y] \,\,\mod \fp,
 \end{align}
 induces a non-trivial map on $\bbP(\fg^{-1} / \fp)$.

 In terms of $\ff(u) = \omega_u(\inf(\cG,\omega))$, we have $0 \neq X \in \ff^0(u)$, but we need more precise information, i.e.\ to rule out $X \in \ff^i(u)$ for $i > 0$.  Given \eqref{E:KH}, then from \eqref{E:SA} and Lemma \ref{l.PR}, we conclude that
 \begin{align}
 \fs_{\geq 0}(u) \subset \fa^{\kappa_H(u)}_{\geq 0} = \fa^{\kappa_H(u)}_0 = \fann(\kappa_H(u)) \subset \fg_0.
 \end{align}
 Hence, we must have $0 \neq \tgr_0(X) \in \fs_0(u) \subset \fg_0$.  Since $\fg_0 \cong \fgl_2$ and $\fg_{-1}$ is the standard $\fgl_2$-representation, we have that $\ad_{\tgr_0(X)}|_{\fg_{-1}}$ is nontrivial.  The induced map on $\bbP(\fg^{-1} / \fp)$ is trivial if only if $\ad_{\tgr_0(X)}|_{\fg_{-1}}$ is a multiple of the identity element $\id_{\fg_{-1}}$, which we can identify with negative of the grading element $\sfZ \in \fz(\fg_0)$.  But $\sfZ \not\in \fann(\kappa_H(u))$ (since $H_2(\fg_+,\fg)$ is positively graded, by regularity).

\end{document}